\documentclass[a4paper,reqno]{amsart}

\usepackage[T1]{fontenc}
\usepackage[utf8x]{inputenc}
\usepackage[english]{babel}

\usepackage{times}
\usepackage{amsmath,  amssymb,slashed,url,bm,upgreek,amsthm,esint}
\usepackage{graphicx,enumerate}
\usepackage[hypertexnames=false]{hyperref}
\newcounter{intro}

\newtheorem{theo}[intro]{Theorem}

\newtheorem{thm}{Theorem}[section]
\newtheorem{lem}[thm]{Lemma}
\newtheorem{prop}[thm]{Proposition}
\newtheorem{cor}[thm]{Corollary}
\newtheorem{defi}[thm]{Definition}
\newtheorem{rem}[thm]{Remark}

\newcommand{\cref}[1]{Corollary~\ref{#1}}

\newcommand{\lref}[1]{Lemma~\ref{#1}}
\newcommand{\pref}[1]{Proposition~\ref{#1}}

\newcommand{\tref}[1]{Theorem~\ref{#1}}

\def\cB{\mathcal B}
\def\cC{\mathcal C}

\def\cH{\mathcal H}
\def\cK{\mathcal K}

\def\cU{\mathcal U}
\def\cV{\mathcal V}
\def\cW{\mathcal W}

\def\R{\mathbb R}

\def\bfA{\textbf{A}}
\def\bfB{\textbf{B}}

\DeclareMathOperator{\un}{\mathbf{1}}
\DeclareMathOperator{\Aun}{A_1}

\DeclareMathOperator{\dist}{d}

\DeclareMathOperator{\capa}{cap}

\DeclareMathOperator{\supp}{supp}

\DeclareMathOperator{\hW}{{\overset{o}{W}}{}^{1,2}}
\DeclareMathOperator{\Wloc}{W_{\rm loc}^{1,2}}
\DeclareMathOperator{\Wo}{W_{\rm o}^{1,2}}

\def\ra{\rangle}
\def\la{\langle}

\newcommand{\loc}{{\rm loc}}

\def\ra{\rangle}
\def\la{\langle}

\title[Boundedness of Schr\"odinger operator in energy space.]{Boundedness of Schr\"odinger operator in energy space.}

\author{Gilles Carron}
\address{G. Carron, Nantes Université, CNRS, Laboratoire de Mathématiques Jean Leray, LMJL, UMR 6629, F-44000 Nantes, France.} 
\email{Gilles.Carron@univ-nantes.fr}
\author{Ma\"el Lansade}
\address{M. Lansade,  Aix-Marseille Université, CNRS, Institut de Mathématiques de Marseille, I2M, UMR 7373, 13453 Marseille Cedex 13, France.} 
\email{mael.lansade@univ-amu.fr}
\begin{document}
\maketitle
\begin{abstract} On a complete weighted Riemannian manifold $(M^n,g,\mu)$ satisfying the doubling condition and the Poincaré inequalities, we characterize the class of function $V$ such that the Schr\"odinger operator $\Delta-V$ maps the homogeneous Sobolev space $\hW(M)$ to its dual space. On Euclidean space, this result is due to Maz'ya and Verbitsky. In the proof of our result, we investigate the weighted $L^2$-boundedness of the Hodge projector.  \end{abstract}
\section{Introduction}
Our main result is a generalization of this result of Maz'ya and Verbitsky \cite{MVacta}:
\begin{thm}\label{thm:MVacta} Let $V$ be a distribution on $\R^n$, $n\ge 3$, then the following properties are equivalent:
\begin{enumerate}
\item there is a positive constant $\bfA$ such that 
\begin{equation}\label{VbounRn}
\forall \varphi\in \cC^\infty_c(\R^n)\colon \left| \langle V, \varphi^2\rangle\right| \le \bfA \int_{\R^n} |d\varphi|^2\,dx,\end{equation}
\item there exists a $1-$form $\theta=\sum_j \theta_jdx_j \in L^2_{loc}$ solving $d^*\theta=V$ and a positive constant $\bfB$:
$$\forall \varphi\in \cC^\infty_c(\R^n)\colon \int_{\R^n} |\theta|^2\varphi^2 dx \le \bfB \int_{\R^n} |d\varphi|^2\,dx.$$
\end{enumerate}
Moreover one can chose $\theta=d\Delta^{-1}V$ and the constants  $\bfA$ and $\sqrt{\bfB}$ are mutually controlled in the sense where if $2)$ holds with constant $\bfB$ then $1)$ holds with $\bfA= 2\sqrt{\bfB}$ and there is a positive constant $c_n$ depending only on $n$ such that if $1)$ holds with constant $\bfA$ then $2)$ holds with $\bfB= c_n\bfA^2.$
\end{thm}

An equivalent formulation of the condition \eqref{VbounRn} is that 
$$\forall \varphi,\phi\in \cC^\infty_c(\R^n)\colon \left| \langle V, \varphi\phi\rangle\right| \le \bfA \|d\varphi\|_{L^2}\,\|d\phi\|_{L^2}.$$

When  $(M^n,g,\mu)$ is a complete weighted Riemannian manifold where $d\mu=e^f d\text{vol}_g$ is a smooth measure, we would like to characterize the distributions $V$ on $M$ for which there exists a positive constant $\bfA$ such that: 
\begin{equation}\label{VbounM}
\forall \varphi,\phi\in \cC^\infty_c(M)\colon \left| \langle V, \varphi\phi\rangle\right| \le \bfA \|d\varphi\|_{L^2_\mu}\|d\phi\|_{L^2_\mu};\end{equation}
where for $\varphi\in\cC^\infty_c(M)$, we define
$$\|d\varphi\|^2_{L^2_\mu}=\int_M |d\varphi|_g^2\,d\mu.$$
A first limitation is that $(M^n,g,\mu)$ must be non-parabolic:

\begin{defi}\label{def:para} A complete weighted Riemannian manifold $(M^n,g,\mu)$ is said to be parabolic if there is a sequence $\chi_\ell\in \cC^\infty_c(M)$ such that
$$\begin{cases}
0\le \chi_\ell\le  1&\text{ every where on }M\\
\lim_{\ell\to+\infty} \chi_\ell= 1&\text{  uniformly on compact set }\\
\lim_{\ell\to+\infty}\int_M |d\chi_\ell|_g^2\, d\mu=0.&\end{cases}$$
A complete weighted Riemannian manifold that is not parabolic is said to be non-parabolic.
\end{defi}
It is well-known that the Euclidean space $\R^n$ is non-parabolic if and only if $n\ge 3$. If $V$ is a distribution satisfying \eqref{VbounM} on a complete parabolic weighted Riemannian manifold, then testing \eqref{VbounM}, with $\varphi$ and  $\phi=\chi_\ell$  (given by Definition \ref{def:para}), we get that $\langle V, \varphi\rangle=0$;  hence the zero distribution is the only one verifying  \eqref{VbounM}. 

If  $(M^n,g,\mu)$ is complete weighted Riemannian manifold, we define its Laplacian $\Delta$ through the Green formula:
$$  \forall \varphi,\phi\in \cC^\infty_c(M)\colon \int_M \langle d\varphi,d\phi\rangle_g\,d\mu=\int_M\varphi\,\Delta\phi\,d\mu,$$
so that our convention is that on the Euclidean space
$\Delta=-\sum_j \partial^2\!/\!\partial x_j^2.$ We will also note $\Delta$ to be the unique self-adjoint extension associated to the symmetric operator
$$\Delta\colon \cC^\infty_c(M)\rightarrow L^2_\mu(M).$$ 
Moreover $(M^n,g,\mu)$ is non-parabolic then one can define the homogeneous Sobolev space
$\hW(M)$ to be  the completion of $\cC^\infty_c(M)$ for the norm $\varphi\mapsto \|d\varphi\|_{L_\mu^2}$ and there is a natural injection $\hW(M) \hookrightarrow \Wloc(M)$. The Green formula implies that $\Delta$ extends as a bounded operator between $\hW(M)$ and its dual $\left(\hW(M)\right)'$, and a distribution $V$  satisfies \eqref{VbounM} if and only if the Schr\"odinger operator $\Delta+V$ extends to a  bounded operator:
$$\Delta+V\colon \hW(M)\rightarrow \left(\hW(M)\right)'.$$

In order to be able to show a generalization of \tref{thm:MVacta}, we will assume that $(M^n,g,\mu)$ satisfies the doubling condition and the Poincaré inequalities. 

\begin{defi}A  complete weighted Riemannian manifold $(M^n,g,\mu)$ is said to satisfy the doubling condition if there are positive constants $\upkappa,\upnu$ such that for any $x\in M$ and any $0<r\le R\colon$
\begin{equation}\label{doubling}
\tag{$\text{D}_{\upkappa,\upnu}$} \mu\left(B(x,R)\right)\le \upkappa \left(\frac R r\right)^\upnu \ \mu\left(B(x,r)\right).
\end{equation}\end{defi}
\begin{rem}\label{nu2}It is well know that if  $(M^n,g,\mu)$ is non-parabolic and satisfy the doubling condition \eqref{doubling} then $\nu>2$.
\end{rem}
\begin{defi}A  complete weighted Riemannian manifold $(M^n,g,\mu)$ is said to satisfy the Poincaré inequalities if there is a positive constant $\uplambda$ such that for any geodesic ball $B$ of radius $r$:
\begin{equation}\label{def:P}\tag{$\text{P}_\uplambda$} \forall\varphi\in \cC^1(B):\int_B \left|\varphi-\varphi_B\right|^2\ d\mu\le \uplambda\, r^2  \int_B \left|d\varphi\right|^2 d\mu,\end{equation}
where $\varphi_B=\fint_B \varphi d\mu$.
\end{defi}

If  the Ricci curvature of $(M^n,g)$ is non negative, then $(M^n,g,\text{vol}_g)$ satisfies these two conditions. We know that there is a smooth positive function (the heat kernel) $P\colon (0,+\infty)\times M\times M\rightarrow \R$ such that for any $f\in L^2_{c}(M):$
$$\left(e^{-t\Delta}f\right)(x)=\int_M P(t,x,y) f(y)\, d\mu(y).$$
 According to \cite{Griharnack,Saloffharnack}, the conjonction of doubling condition \eqref{doubling} and of the Poincaré inequalities  \eqref{def:P} is equivalent to  uniform upper and lower Gaussian estimate of the heat kernel:
$$
\forall x,y\in M,\ \forall t>0\colon  \frac{c}{\mu\left(B(x,\sqrt{t})\right)}e^{- \frac{d^2(x,y)}{ct}}\le P(t,x,y)\le \frac{C}{\mu\left(B(x,\sqrt{t})\right)}e^{-\frac{d^2(x,y)}{5t}}.
$$

Our main result is then the following:

\begin{theo}\label{thm:MV}Let  $(M,g,\mu)$ be a complete non-parabolic weighted Riemannian manifold satisfying the the doubling condition \eqref{doubling} and the Poincaré inequalities $\eqref{def:P}$. A distribution $V$ satisfies the inequality \eqref{VbounM} if and only of there is a $1-$form $\theta\in L^2_{loc}(T^*M)$ solving
$d^*_\mu\theta=V$ and such that
\begin{equation}\label{ThetabounM}\forall \varphi\in \cC^\infty_c(M)\colon \int_{M} |\theta|^2\varphi^2\, d\mu \le \bfB \int_{M} |d\varphi|_g^2\,dx.\end{equation}
Moreover the constants  $\bfA$ and$\sqrt{\bfB}$ are mutually controlled: if \eqref{ThetabounM} holds with constant $\bfB$ then \eqref{VbounM} holds with $\bfA= 2\sqrt{\bfB}$ and there is a positive constant $c$ depending only of $\upkappa,\upnu,\uplambda$ such that if \eqref{VbounM} holds with constant $\bfA$ then \eqref{ThetabounM} holds with $\theta=d\Delta^{-1}V$ and $\bfB= c\bfA^2.$

\end{theo}

We recall that the equation $d^*_\mu\theta=V$ is equivalent to the fact that
$$\forall \varphi \in \cC^\infty_c(M)\colon \int_M \langle \theta,d\varphi\rangle_g\, d\mu=\langle V, \varphi\rangle.$$

Our proof follows the original arguments of Maz'ya and Verbitsky but for this purpose, we prove some new results about the Hodge projector. To describe this operator let  $(M,g,\mu)$ be a complete non-parabolic weighted Riemannian manifold, it has a positive Green kernel defined for $x\not=y$ by:
$$G(x,y)=\int_0^{+\infty} P(t,x,y)dt.$$
When $\beta\in \cC^\infty_c(T^*M)$  is a smooth compactly supported $1-$form, one defines $\Pi\beta$ by
$$\Pi\beta=d\varphi\text{ where } \varphi(x)=\int_M G(x,y) d^*_\mu\beta(y)\, d\mu(y).$$
The operator $\Pi$ extends to a bounded operator on $L^2_\mu(T^*M)$ and this extension, which is also noted $\Pi$, is the orthogonal projection on the closure of   $d\cC^\infty_c(M)\subset L^2_\mu(T^*M)$. The operator $\Pi=d\Delta^{-1}d^*_\mu$ is called the Hodge projector.

When $\rho$ is a non negative and non trivial Radon measure with compact supported then the function $h_\rho$ defined by
\begin{equation}\label{defh}
h_\rho(x)=\int_M G(x,y)d\rho(y)\end{equation} is  a positive superharmonic function. A crucial result for proving \tref{thm:MV} is the following

\begin{theo}\label{thm:BoundH}Let  $(M,g,\mu)$ be a complete non-parabolic weighted Riemannian manifold, then for any $\delta\in (-1,1)$ and any non negative Radon measure with compact supported $\rho$ and associated superharmonic function $h_\rho$ defined by \eqref{defh}:
$$
\forall \beta \in \cC^\infty_c(T^*M)\colon\int_M \left|\Pi(\beta)\right|^2 \, h_\rho^\delta d\mu\le \left(\frac{1+|\delta|}{1-|\delta|}\right)^2\  \int_M \left|\beta\right|^2 \, h_\rho^\delta d\mu.$$
Moreover if $(M,g,\mu)$ satisfies  the doubling condition \eqref{doubling} and the Poincaré inequalities $\eqref{def:P}$ then there are some $\delta_+>1$ and $C>0$ depending only on $\upkappa,\upnu,\uplambda$ such that  for any non negative Radon measure with compact supported $\rho$ and associated superhamronic function $h_\rho$ defined by \eqref{defh}: 
$$
\forall \beta \in \cC^\infty_c(T^*M)\colon\int_M \left|\Pi(\beta)\right|^2 \, h_\rho^{\delta_+} d\mu\le C\  \int_M \left|\beta\right|^2 \, h_\rho^{ \delta_+} d\mu.$$
\end{theo}

Classical result on singular integral implies that on $\R^n$, the second conclusion of \tref{thm:BoundH} holds with $\delta_+$ being any real number strictly less than $n/(n-2)$; this was an important argument in the proof of Maz'ya-Verbitsky (see \cite[page 284]{MVacta}).
The study of the boundedness of the Hodge operator on $L^p$ spaces in relationship with the boundedness of the Riesz transform on $L^p$ spaces is a well  developed subject (see for instance  \cite[subsection 2.3]{ACPisa}); the question of the boundedness of the Hodge operator on weighted  $L^p$ spaces has been extensively studied by Auscher and Martell  \cite{AMadv, AMZ}. These results do not imply directly the second assertion of \ref{thm:BoundH}. The starting point is the universal boundedness of the Hodge projector on  
$L^2_{h^\delta\mu}(T^*M)$ for $\delta\in (-1,1)$ and we apply a result of Auscher-Martell \cite{AMadv} on the weighted Riemannian manifold $(M,g,h^\delta \mu)$; in order to be in position of using this result, we have to prove that those functions $h_\rho^\delta$ are  $\Aun\!\!-$weights. 

Our universal boundedness of the Hodge projector on some weighted $L^2$ space could be a first step toward a more general investigation of the $L^p$-weighted boundedness of the Hodge operator .

Another important result  for proving \tref{thm:MV} is a a decay estimate on $\Pi(\beta)$ when  $\beta\in \cC^\infty_c(T^*M)$, more precisely, we show that there is some $p>1$ such that 
$$\int_{M\setminus B(o,1)} G^{-p}(o,x) |\Pi(\beta)|^2(x)\, d\mu(x)<\infty.$$

The next section is devoted to the proof of the first conclusion of  \tref{thm:BoundH} and we will recall and prove some useful results about Green kernel and superharmonic functions; in section 3, we present various equivalent conditions  for an estimate of the type
$$\forall \varphi\in \cC^\infty_c(M)\colon \int_{M} q\varphi^2\, d\mu \le C \int_{M} |d\varphi|_g^2\,dx,$$
when $q$ is a non-negative locally integrable function.  One of the first results in the Euclidean case is due to Fefferman-Phong \cite{FP} and there were then many works dealing with this questions in  Euclidean spaces or on homogeneous space (see for instance \cite{Fefferman,KS,Mazya,MVarkiv,PW,SawyerW,SWZ,Sch}; the Riemannian setting was adressed recently in \cite{Lansade}. The second conclusion of  \tref{thm:BoundH} is proven in section 4 and  \tref{thm:MV} is proven in section 5. 
In an appendix, we explain how an elegant Euclidean argument of Verbitksy \cite{Verbsurvol} gives  trace inequalities on doubling metric measure spaces.

\noindent \textbf{Acknowledgments:} We are partially supported by the ANR grants ANR-18-CE40-0012: RAGE.
 and ANR-17-CE40-0034: CCEM and we  thank the Centre Henri Lebesgue ANR-11-LABX-0020-01 for creating an attractive mathematical environment. 
\section*{Some notations}

When $(M,g,\mu)$ is a weighted complete Riemannian manifold then for any Borel set $\cB\subset M$, we will note
$\un_\cB$ the characteristic function of $\cB$; if moreover $\mu(\cB)\not=0$ and if $\varphi$ is an integrable function on $\cB$, the mean of $\varphi$ on $\cB$ will be noted by
$$\fint_\cB \varphi=\fint_\cB \varphi d\mu=\frac{1}{\mu(\cB)}\int_{\cB} \varphi d\mu;$$
the subscript ${}_c$ will indicate a space made of compactly support functions or forms, for instance $L^2_{c}(M), \cC^\infty_c(M),\cC^\infty_c(T^*M)$; and the space of locally $L^2$ fonction (resp. locally $W^{1,2}$ functions) is noted $L^2_{loc}(M)$ (resp. $W_{loc}^{1,2}$).

When $B=B(x,r)$ is a geodesic ball and $\theta>0$ we will define $\theta B:=B(x,\theta r)$. For each $x\in M$ and $\beta\in T^*_xM$ the norm of $\beta$ is 
$$ |\beta|=|\beta|_g=\sup_{\xi \in T_xM, g_x(\xi,\xi)=1} \beta(\xi)$$ and it induces on $T^*_xM$ a scalar product that will be noted
$\langle\cdot,\cdot\rangle_g$ or $\langle\cdot,\cdot\rangle$, from now we will omit the $g$ subscript, this induces a scalar product on $1-$forms
$$\beta_1,\beta_2\in \cC^\infty_c(T^*M)\colon \langle \beta_1,\beta_2 \rangle=\int_M \langle\ \beta_1(x),\beta_2(x)\rangle_g\, d\mu(x).$$
The associate Hilbert space collecting all square integrable measurable section of $T^*M$ will be noted $L^2_\mu(T^*M)$.
If $\cU\subset M$ is a bounded open subset $\Wo(\cU)$ is a the closure of $\cC^\infty_c(\cU)$  in  $\Wloc(M)$, it is also the completion of $\cC^\infty_c(\cU)$ for the topology induced by the norm $\varphi\mapsto \sqrt{\|\varphi\|_{L^2_\mu}^2+\|d\varphi\|_{L^2_\mu}^2}$.

 \section{Weighted inequalities given by equilibrium potentials}
In this section, we collect some well-known facts and others useful technical results that leads to a very general weighted boundedness result for the Hodge projector. Thorough all this section, $(M,g,\mu)$ is a weighted complete Riemannian manifold that is assumed to be non parabolic.
\subsection{Capacity and equilibrium potential} The non parabolicity implies the following properties \cite{Ancona,BamsGri}:
\begin{prop}\label{prop:capa}
\begin{enumerate}[a)]
\item For each bounded open set $\cU\subset M$, there is a positive constant $C_\cU$ such that:
$$\forall \varphi\in \cC^\infty_c(M)\colon \int_{\cU} \varphi^2d\mu\le C_\cU\, \int_M |d\varphi|^2d\mu.$$
\item $(M,g,\mu)$ admits a positive minimal Green kernel $G\colon M\times M \rightarrow \R_+\cup\{+\infty\}$ such that for any $f\in \cC^\infty_c(M)$, a solution of the equation $\Delta u=f$  is given by
 \begin{equation}\label{eq:sol}
 u(x)=\int_M G(x,y)f(y)d\mu(y).\end{equation}
 \item For any bounded open subset $\cV\subset M$, the Green kernel of the  Dirichlet-Laplace operator on $\cV$ is noted $G^{\cV}\colon \cV\times \cV\rightarrow \R_+\cup\{+\infty\}$. Then we have that
$$G(x,y)=\lim_{\cV\to M} G^{\cV}(x,y)=\sup_\cV G^{\cV}(x,y).$$
\end{enumerate}
\end{prop}

The function $G_o(y)=G(o,y)$ is called the Green function with pole at $o\in M$; it is a smooth positive harmonic function on $M\setminus \{o\}.$
\pref{prop:capa}-a) implies that if $\hW(M)$ is the completion of $\cC^\infty_c(M)$ for the norm $\varphi\mapsto \|d\varphi\|_{L_\mu^2}$ then the injection $\cC^\infty_c(M) \hookrightarrow \Wloc(M)$ extends continuous and provides a natural injection $\hW(M) \hookrightarrow \Wloc(M)$. Moreover for any $f\in L^2_\mu$ with compact support there is a unique $u\in\hW(M)$ solving the equation $\Delta u=f$ and $u$ is given by the  formula  \eqref{eq:sol}. 

It also implies that for any non empty bounded open subset $\cU\subset M$, its {\it capacity} 
\begin{equation}\label{def:capa}
 \capa(\cU)=\inf\left\{ \int_M |d\varphi|^2d\mu, \text{ such that } \varphi\in \cC^\infty_c(M)\text{ and } \varphi=1\text{ on } \cU\right\}\end{equation} is strictly positive and there is a unique
$h\in \hW(M)$ such that
$$\capa(\cU)=\int_M |dh|^2d\mu\text{ and } h=1\text{ on } \cU;$$
moreover $0<h\le 1$. This function $h$ is called the {\it equilibrium potential} of  $\cU$, it is a superharmonic function and there is a positive Radon measure  (called the {\it equilibrium measure}) $\nu_{\cU}$ supported in $\cU$ such that 
$$\forall x\in M\colon \,h(x)=\int_M G(x,y)d\nu_\cU(y);$$
the equilibrium measure of $\cU$ is also given by the equality (in the distributional sense)
$$d\nu_\cU=\Delta h.$$
For instance if $\cU$ has smooth boundary and if $\vec{\rm n}\colon\partial \cU\rightarrow TM$ is the inward unit normal vector fields and $\sigma$ the induced measure on $\partial \cU$ then the equilibrium measure of $\cU$ is given by
$$\nu_\cU(K)=\int_{\partial \cU\cap K} \frac{\partial h}{\partial \vec{\rm n}}\, d\sigma.$$ 
We recall:
\begin{defi}A function $f\colon M\rightarrow \R$ is said to be superharmonic if for any non negative $\varphi\in \cC^\infty_c(M)\colon$
$$\int_M f\Delta\varphi \,d\mu\ge 0.$$
\end{defi}
Notice that for any bounded $f \colon M\rightarrow \R$ and $\varphi\in\cC^\infty_c(M)$  we have that 
$$\int_M f\Delta\varphi\,d\mu=\lim_{t\to 0} \int_M f\,\frac{\varphi-e^{-t\Delta}\varphi}{t}\,d\mu=\lim_{t\to 0} \int_M \frac{f-e^{-t\Delta}f}{t}\,\varphi\,d\mu.$$
Hence 
\begin{lem}\label{lem:critsuperh} A positive bounded function $f\colon M\rightarrow \R$ is  superharmonic if and only if for any $t>0\colon$
$$e^{-t\Delta}f\le f.$$
\end{lem} 
Notice that the result still holds if we assume that  
$$\forall t>0,\forall x\in M\colon\ \int_M P(t,x,y)f(y)\,d\mu(y)<+\infty.$$
\subsection{Some properties of  equilibrium potentials}
\begin{prop} \label{prop:htau}\begin{enumerate}[a)] 
\item For any $o\in M$ and any $\tau>1/2$, the Green function with pole at $o$ satisfies $\min\left(G_o^\tau,1\right)\in \hW(M)$.
\item For any $\tau>1/2$ and any  bounded open set  $\cU\subset M$ with equilibrium potential $h$, we have $h^\tau\in \hW(M)$ moreover
$$\int_M \left|dh^\tau\right|^2d\mu\le \frac{\tau^2}{2\tau-1} \capa(\cU).$$
\end{enumerate}
\end{prop}
\proof Let $\Omega_\ell$ an increasing sequence of bounded open subset with smooth boundary exhausting $M$.  For large enough $\ell$, we have that $o\in \Omega_\ell$ and let $g_\ell=G^{\Omega_\ell}(o,\cdot)$.
It is enough to show that there is a constant $C$, independant of $\ell$, such that
$$\int_{\{g_\ell< 1\} } \left|dg_\ell^\tau\right|^2d\mu\le C.$$
Using the coaera formula, we get that
\begin{align*}
\int_{\{g_\ell< 1\} } \left|dg_\ell^\tau\right|^2d\mu&=\tau^2\int_{\{g_\ell< 1\} }g_\ell^{2\tau-2} \left|dg_\ell\right|^2d\mu\\
&=\tau^2\int_0^1 x^{2\tau-2} \left(\int_{ g_\ell=x} \left|dg_\ell\right|d\sigma_x\right)dx
\end{align*}
But for any regular value $x$ of $g_\ell$, the Green formula indicates that the integral  $\int_{ g_\ell=x} \left|dg_\ell\right|d\sigma_x$ does not depends on $x$:
$$0=\int_{y<g_\ell<x} \Delta g_\ell d\mu=\int_{ g_\ell=y} \left|dg_\ell\right|d\sigma_y-\int_{ g_\ell=x} \left|dg_\ell\right|d\sigma_x.$$
Letting $x\to+\infty$ and using the asymptotics of the Green kernel around $o$, we find that 
$$\int_{ g_\ell=y} \left|dg_\ell\right|d\sigma_y=1.$$ 
Hence for any $\tau>1/2$:
$$\int_{\{g_\ell< 1\} } \left|dg_\ell^\tau\right|^2d\mu=\tau^2\int_0^1 x^{2\tau-2} dx=\frac{\tau^2}{2\tau-1}.$$
The proof concerning an equilibrium potential is exactly the same. For $\ell$ large enough, $\bar\cU\subset \Omega_\ell$ and we introduce $h_\ell\in \Wo(\Omega_\ell)$, the equilibrium potential relative to $\Omega_\ell$, it is characterized by 
$$\int_M |dh_\ell |^2d\mu=\inf\left\{ \int_M |d\varphi|^2d\mu, \text{ such that } \varphi\in \cC^\infty_c(\Omega_\ell)\text{ and } \varphi=1\text{ on } \cU\right\}.$$
Then $h_\ell$ is harmonic on
$\Omega_\ell\setminus \bar \cU$ and $\Omega_\ell\setminus \bar \cU=h_\ell^{-1}\left( (0,1)\right)$.  Similarly for every regular value $x\in (0,1)$ of $h_\ell$, the integral
$\int_{ h_\ell=x} \left|dh_\ell\right|d\sigma_x$ does not depend on $x$ and
$$\int_M |dh_\ell |^2d\mu=\int_0^1  \left(\int_{ h_\ell=x} \left|dh_\ell\right|d\sigma_x\right)dx.$$
So that for a.e. $x\in (0,1)$:
$$\int_{ h_\ell=x} \left|dh_\ell\right|d\sigma_x=\int_M |dh_\ell |^2d\mu.$$
and if $\tau>1/2$ then

$$\int_M |dh^\tau_\ell |^2d\mu=\tau^2\int_0^1 x^{2\tau-2} \left(\int_{ h_\ell=x} \left|dh_\ell\right|d\sigma_x\right)dx=\frac{\tau^2}{2\tau-1}\int_M |dh_\ell |^2d\mu.$$
It is easy to show that the sequence $\left(h_\ell\right)_\ell$ converges to $h$ in $\hW(M)$ and this estimate implies that for $\tau>1/2$, $h^\tau_\ell$ converges weakly to $h^\tau$ in $\hW(M)$; hence the result.
\endproof
\begin{prop}\label{prop:superh} For any $\tau\in [0,1]$ and any $o\in M$; $G_o^\tau$ is superharmonic. Similarly for any $\tau\in [0,1]$ and  any bounded open subset $\cU$ with  equilibrium potential $h$, $h^\tau$ is superharmonic.\end{prop}
\proof
We know that
$G_o(x)=\int_0^\infty P(\tau,o,x)d\tau$ hence for any $t>0\colon$
$$\left(e^{-t\Delta}G_o\right)(x)=\int_t^\infty P(\tau,o,x)d\tau\le G_o(x).$$ Hence $G_o$ is superharmonic then using H\"older inequality and that 
$e^{-t\Delta}\un_M\le\un_M$, we have for any $\tau\in [0,1]\colon$ 
$$e^{-t\Delta}G^\tau_o\le \left(e^{-t\Delta}G_o\right)^\tau\le G_0^\tau.$$ Hence $G_o^\tau$ is also superharmonic. The proof of the second assertion is identical.
\endproof
\begin{rem} Using Jensen inequality, we can show that for any superharmonic  function $h$ and any concave function $\Phi$, then $\Phi(h)$ is superharmonic.
\end{rem}
\subsection{Weighted Hardy type inequalities}The next result is a consequence of the argumentation presented in \cite{Hardy}.
\begin{prop} If $h\colon M\rightarrow \R_+$ is the equilibrium potential of some bounded open set, then for any $\delta<1$ then  the following general  Hardy type inequality holds:
\begin{equation}\label{hardyh}\forall \varphi\in \cC^\infty_c(M)\colon  \left(\frac{\delta-1}{2}\right)^2\, \int_M \frac{|dh|^2}{h^2} \,\varphi^2 \, h^\delta d\mu\le  \int_M |d\varphi|^2 \, h^\delta d\mu.\end{equation}
\end{prop} 
\proof Let  $\phi\in \cC^\infty_c(M)$ and define $\varphi=h^{\frac{1-\delta}{2}}\phi$, we compute:
$$|d\varphi|^2=h^{1-\delta}|d\phi|^2+\left(\frac{\delta-1}{2}\right)^2\frac{|dh|^2}{h^2} \, \varphi^2+(1-\delta)\la dh,d\phi\ra\ h^{-\delta}\phi.$$
Then we get 
$$\int_M |d\varphi|^2 \, h^\delta d\mu\ge  \left(\frac{\delta-1}{2}\right)^2\int_M \frac{|dh|^2}{h^2} \, \varphi^2h^\delta d\mu+\frac{1-\delta}{2}\int_M\la dh,d\phi^2\ra\,d\mu.$$
But $h$ is superharmonic hence if $\delta<1$ we get that
$$\int_M |d\varphi|^2 \, h^\delta d\mu\ge  \left(\frac{\delta-1}{2}\right)^2\int_M \frac{|dh|^2}{h^2} \, \varphi^2h^\delta d\mu.$$
Then by approximation, we get that this  weighted Hardy inequality is valid for any $\varphi\in\Wloc(M)$ with compact support.
\endproof
\subsection{Weighted $L^2$-boundedness of the Hodge projector.}
Recall that the space of $L^2$ $1-$forms has the following {\it Hodge}  orthogonal decomposition:
$$L^2_\mu(T^*M)=\cH^1(M,\mu)\oplus \overline{d\cC^\infty_c(M)}\oplus \overline{d^*_\mu\cC^\infty_c(\Lambda^2T^*M)},$$
where the closure are taken with respect to the $L^2_\mu$ topology; the operator $$d_\mu^*\colon \cC^\infty_c(\Lambda^2T^*M)\rightarrow \cC^\infty_c(T^*M)$$ is the formal adjoint of exterior differential
operator $d\colon \cC^\infty_c(T^*M)\rightarrow \cC^\infty_c(\Lambda^2T^*M):$
$$\forall \alpha\in \cC^\infty_c(T^*M),\beta\in \cC^\infty_c(\Lambda^2T^*M):\ 
\int_M \la d\alpha,\beta\ra\, d\mu=\int_M \la \alpha,d_\mu^*\beta\ra\, d\mu;$$
and $$\cH^1(M,\mu)=\left\{ \alpha\in L^2_\mu(T^*M), d\alpha=0\text{ and }d_\mu^*\alpha=0\right\}.$$
The Hodge projector $\Pi\colon L^2_\mu(T^*M)\rightarrow L^2_\mu(T^*M)$ is the $L^2_\mu$-projector on $ \overline{d\cC^\infty_c(M)}$. As we assumed that $(M,g,\mu)$ is non parabolic for any
$\alpha\in \cC^\infty_c(T^*M)$, we have
$$\Pi\alpha=d\varphi$$ where
$\varphi$ solves the equation $\Delta \varphi=d^*_\mu\alpha$ that is to say:
$$\varphi(x)=\int_M G(x,y)d^*_\mu\alpha(y)d\mu(y)=\int_M \la d_yG(x,y),\alpha(y)\ra\,d\mu(y).$$

\begin{thm}\label{thm:L2Whodge}
If  $h\colon M\rightarrow \R_+$ is an equilibrium potential and $\delta\in (-1,1)$ then for any $\beta \in \cC^\infty_c(T^*M)\colon$
$$\int_M \left|\Pi(\beta)\right|^2 \, h^\delta d\mu\le \left(\frac{1+|\delta|}{1-|\delta|}\right)^2\  \int_M \left|\beta\right|^2 \, h^\delta d\mu.$$
In particular $\Pi\colon \cC^\infty_c(T^*M)\rightarrow L^2_{h^\delta\mu}(T^*M)$ has a bounded extension
$\Pi\colon L^2_{h^\delta\mu}(T^*M) \rightarrow L^2_{h^\delta\mu}(T^*M)$.
\end{thm}
\proof
As $\Pi$ is bounded and selfadjoint on $ L^2_\mu(T^*M)$, by duality it is enough to show the result for $\delta\in (0,1).$ So let $\delta\in (0,1)$ and $\beta\in \cC^\infty_c(T^*M)$ and let
$\varphi\in \hW(M)$ be the solution of the equation $\Delta \varphi=d_\mu^*\beta$. 
We first show that 
\begin{equation}\label{eq:extsh} \int_M |d\varphi|^2\,h^\delta d\mu\le \int_M \varphi\Delta\varphi\,h^\delta d\mu.\end{equation}

Consider a sequence $(\varphi_\ell)$ made of smooth compactly supported functions that converges toward $\varphi$ in $\hW(M)$. Because $h^\delta$  is superharmonic, we get that for all $\ell$
$$\int_M h^\delta \Delta\varphi_\ell^2\, d\mu\ge 0.$$
But $\frac12\Delta \varphi_\ell^2=\varphi_\ell\Delta\varphi_\ell-\left|d\varphi_\ell\right|^2$ hence
 \begin{equation}\label{eq:extsh1}\int_M |d\varphi_\ell|^2\,h^\delta d\mu\le \int_M \varphi_\ell\Delta\varphi_\ell\,h^\delta d\mu.\end{equation}

As $h^\delta$ is a bounded positive function, we get that 
 \begin{equation}\label{eq:extsh2}\int_M |d\varphi|^2\,h^\delta d\mu=\lim_\ell\int_M |d\varphi_\ell|^2\,h^\delta d\mu.\end{equation}
Using
\begin{enumerate}[i)]
\item $\displaystyle d\left(\varphi_\ell h^\delta\right)=h^\delta d\varphi_\ell+\delta \frac{dh}{h} \varphi_\ell h^\delta;$
\item the Hardy inequality \eqref{hardyh}:
$$\forall \phi\in \cC^\infty_c(M)\colon  \, \int_M  \frac{|dh|^2}{h^2} \,\phi^2 \, d\mu\le  4\int_M |d\phi|^2 \,   d\mu;$$ 
\item $h$ is a  positive function and takes value in $(0,1]$,
\end{enumerate}
we deduce that for any $\ell,k$:
$$\left\|d\left(\varphi_\ell h^\delta-\varphi_k h^\delta\right)\right\|^2_{L^2_\mu}\le \left(8\delta+2\right) \left\|d\left(\varphi_\ell -\varphi_k \right)\right\|^2_{L^2_\mu}.$$ hence $h^\delta\varphi\in \hW(M)$ and  
 $$\lim_{\ell \to +\infty} \left\|d\left(\varphi_\ell h^\delta-\varphi h^\delta\right)\right\|_{L^2_\mu}=0.$$ So that we get: 
 $$\lim_\ell \int_M \varphi_\ell\Delta\varphi_\ell\,h^\delta d\mu=\lim_\ell \int_M \la d\varphi_\ell, d\left(\varphi_\ell\,h^\delta\right)\ra d\mu= \int_M \la d\varphi, d\left(\varphi\,h^\delta\right)\ra d\mu.$$
 But by definition we know that 
 $$\forall \phi\in \hW(M)\colon \int_M \la d\varphi,d\phi\ra\, d\mu=\int_M d_\mu^*\beta\, \phi\, d\mu.$$
 Hence 
 \begin{equation}\label{eq:extsh3}\lim_\ell \int_M \varphi_\ell\Delta\varphi_\ell\,h^\delta d\mu=\int_M d_\mu^*\beta\, \left(\varphi\,h^\delta\right)\, d\mu=\int_M \Delta\varphi\, \left(\varphi\,h^\delta\right)\, d\mu.\end{equation}

The inequalities \eqref{eq:extsh1}, \eqref{eq:extsh2} \eqref{eq:extsh3} implies the inequalities \eqref{eq:extsh}. And using $\Delta\varphi= d_\mu^*\beta$, we deduce that
$$\int_M |d\varphi|^2\,h^\delta d\mu\le \int_M d_\mu^*\beta\,\varphi\,h^\delta d\mu.$$
Integrating by parts, we get that
$$ \int_M d_\mu^*\beta\,\varphi\,h^\delta d\mu= \int_M\la\beta,d\varphi\ra\,h^\delta d\mu+ \delta\int_M \la \beta,dh\ra \varphi\,h^{\delta-1}\, d\mu.$$
Using the Cauchy-Schwartz inequality, we obtain 
$$\left|\int_M\la\beta,d\varphi\ra\,h^\delta d\mu\right|\le \left(\int_M |\beta|^2 \,h^\delta d\mu\right)^{\frac12}\left(\int_M|d\varphi|^2 \,h^\delta d\mu\right)^{\frac12}.$$
And  the weighted Hardy inequality \eqref{hardyh} implies that 
\begin{align*}
\left|\int_M \la \beta,dh\ra \varphi\,h^{\delta-1}\, d\mu\right|&\le\left(\int_M |\beta|^2 \,h^\delta d\mu\right)^{\frac12}\left(\int_M \frac{|dh|^2}{h^2} |\varphi|^2 \,h^\delta d\mu\right)^{\frac12}\\
&\le  \frac{2}{1-\delta}\left(\int_M |\beta|^2 \,h^\delta d\mu\right)^{\frac12}\left(\int_M|d\varphi|^2 \,h^\delta d\mu\right)^{\frac12}.
\end{align*}
So that we eventually obtain
$$\int_M |d\varphi|^2\,h^\delta d\mu\le\left(1+\frac{2\delta}{1-\delta}\right) \left(\int_M |\beta|^2 \,h^\delta d\mu\right)^{\frac12}\left(\int_M|d\varphi|^2 \,h^\delta d\mu\right)^{\frac12}.$$
\endproof
\section{Trace type inequalities}
\subsection{}
In  this section, we collect some general facts about the validity of trace inequality of the type
\begin{equation}\label{trace1}
\forall \varphi\in \cC^\infty_c(M)\colon \ \int_M q\varphi^2\, d\mu\le C_1\int_M |d\varphi|^2\,d\mu,\end{equation}
where $q$ is a non negative locally integrable function. In order to have the existence of at least one non trivial potential $q$ such that \eqref{trace1} holds, we must suppose that $(M,g,\mu)$ is non parabolic and in this case using equality :
$$\forall \varphi\in \cC^\infty_c(M)\colon\ \int_M |d\varphi|^2\,d\mu=\la\varphi,\Delta\varphi\ra_{L^2_\mu}=\ \int_M |\Delta^{\frac12}\varphi|^2\,d\mu$$
we know that the operator $\Delta^{\frac12}$ defined with the spectral theorem extends to an isometry
$$\Delta^{\frac12}\colon \hW(M)\rightarrow L^2_\mu(M),$$ with inverse $\Delta^{-\frac12}\colon  L^2_\mu(M)\rightarrow \hW(M).$ Hence the above condition \eqref{trace1} is equivalent to the fact that $\sqrt{q}\Delta^{-\frac12}$ (or its adjoint $\Delta^{-\frac12}\sqrt{q}$) has a bounded extension on  $L^2_\mu(M)$ with operator norm satisfying
$$\left\|\sqrt{q}\Delta^{-\frac12}\right\|^2\le C_1.$$

\subsection{A general result}
The following result is well known in the Euclidean setting and is mostly due to Maz'ya and Vertbisky \cite{Mazya,MVarkiv}, however it is folklore that the proofs can be easily adapted in a much more general setting.
\begin{thm}\label{thm:trace} Let $(M,g,\mu)$ be a complete non-parabolic weighted Riemannian manifold and $q$ be a non negative, locally integrable function then the following properties are equivalents
\begin{enumerate}[i)]
\item there is a constant $C_1$ such that \eqref{trace1} holds;
\item the operator $\sqrt{q}\Delta^{-\frac12}$ has a bounded extension $\sqrt{q}\Delta^{-\frac12}\colon L^2_\mu(M)\rightarrow L^2_\mu(M);$
\item there is a constant $C_3$ such that for any bounded open set $\cU\subset M\colon$
\begin{equation}\label{trace3}\int_{\cU} qd\mu\le C_3 \capa(\cU);\end{equation}
\item there is a constant $C_4$ such that for any bounded open set $\cU\subset M\colon$
\begin{equation}\label{trace4}\int_M \left|\Delta^{-\frac12}\left(q\un_{\cU}\right)\right|^2d\mu\le C_4\, \int_{\cU} q\,d\mu;\end{equation}
\item there is a constant $C_5$ such that for any bounded open set $\cU\subset M\colon$
\begin{equation}\label{trace5}\int_M \left|\Delta^{-\frac12}\left(q\un_{\cU}\right)\right|^2d\mu\le C_5^2\, \capa(\cU).\end{equation}
\end{enumerate}
Moreover  $\left\|\sqrt{q}\Delta^{-\frac12}\right\|^2$ and $C_1,C_3,C_4,C_5$ are mutually controlled.
\end{thm}
\proof
We have already explained that{ \bf   i)$\Leftrightarrow$ ii)}.\par
{Proof of  \bf  ii)$\Rightarrow$ iv).} Testing the $L^2_\mu$-boundedness of $\Delta^{-\frac12}\sqrt{q}=\left(\sqrt{q}\Delta^{-\frac12}\right)^*=T^*$ with $f=\sqrt{q}\un_\cU$, we get 
$$\int_M \left|\Delta^{-\frac12}\left(q\un_{\cU}\right)\right|^2d\mu\le \left\|T\right\|^2\, \int_M\left(\sqrt{q}\un_{\cU}\right)^2d\mu,$$
Hence we get that that $\textbf{iv)}$ holds with $C_4= \left\|\sqrt{q}\Delta^{-\frac12}\right\|^2.$\par

{\bf  Proof of  iv)$\Rightarrow$ v).} We notice that \eqref{trace4} is equivalent to
$$\forall \cU,\forall f\in L^2_\mu\colon\la\left(q\un_\cU\right),\Delta^{-\frac 12}f\ra_{L^2_\mu}= \la\Delta^{-\frac 12}\left(q\un_\cU\right),f\ra_{L^2_\mu}\le \sqrt{C_4} \sqrt{\int_{\cU}qd\mu\ }\|f\|_{L^2_\mu}.$$
That is to say that 
$$\forall \cU,\forall \varphi\in \hW(M)\colon\la q\un_\cU,\varphi\ra_{L^2_\mu}\le  \sqrt{C_4} \sqrt{\int_{\cU}qd\mu\ }\|d\varphi\|_{L^2_\mu}.$$
Then using the definition of the capacity we get that 
$$\sqrt{\int_{\cU} qd\mu}\le \sqrt{C_4 \capa(\cU)}$$ and reporting this estimate in \eqref{trace4} we obtain the inequality \eqref{trace5} with $C_5= C_4.$
\par
{\bf  v)$\Rightarrow$ iii)} The same argumentation yields that  inequality \eqref{trace5} is equivalent to
$$\forall \cU,\forall \varphi\in \hW(M)\colon\la q\un_\cU,\varphi\ra_{L^2_\mu}\le  C_5\sqrt{\capa(\cU)\ }\ \|d\varphi\|_{L^2_\mu}.$$ Hence we obtain the inequality \eqref{trace3} with $C_3= C_5.$

\par
{\bf Proof of  iii)$\Rightarrow$ i).} The argument is the one originally given by Maz'ya. For any $\varphi\in \cC^\infty_c(M)$ and $t>0$ we test \eqref{trace3} on $\cU_t=\{\varphi^2>t\}$ and get that 
$$\int_{\cU_t} qd\mu\le C_3\int_{\{\varphi^2<t^2\}} \frac{\left|d\varphi^2\right|^2}{t^2}d\mu=4C_3\int_{\{\varphi^2<t^2\}} \frac{\varphi^2\left|d\varphi\right|^2}{t^2}d\mu.$$
Then integrating with respect to $t\in (0,+\infty)$ and using the Cavalieri's formula and the Fubini theorem, we get that
\begin{align*}\int_M q\varphi^2\,d\mu&=\int_0^\infty\left(\int_{\cU_t} qd\mu\right)dt\\
&\le 4C_3\int_M \varphi^2\left|d\varphi\right|^2\left(\int^{+\infty}_{\varphi^2}\frac{dt}{t^2}\right)d\mu\\
&=4C_3\int_M \left|d\varphi\right|^2\,d\mu.
\end{align*}
Hence we obtain the inequality \eqref{trace1} with $C_1= 4C_3.$
\endproof
\begin{rem} The proof show that the constants $\left\|\sqrt{q}\Delta^{-\frac12}\right\|^2$ and $C_1,C_3,C_4,C_5$ are mutually controlled in the following way:
\begin{itemize}
\item If $\textbf{i)}$ holds with constant $C_1$ then $\textbf{ii)}$ holds with$ \left\|\sqrt{q}\Delta^{-\frac12}\right\|^2\le C_1$;
\item If $\textbf{ii)}$ then $\textbf{iv)}$ with constant $C_4=\left\|\sqrt{q}\Delta^{-\frac12}\right\|^2$;
\item If $\textbf{iv)}$ holds with constant $C_4$ then $\textbf{v)}$ holds with constant $C_5=C_4$;
\item If $\textbf{v)}$ holds with constant $C_5$ then $\textbf{iii)}$ holds with constant $C_3=C_5$;
\item If $\textbf{iii)}$ holds with constant $C_3$ then $\textbf{i)}$ holds with constant $C_1=4C_3$.
\end{itemize}

\end{rem}
\subsection{With the relative Faber-Krahn inequality}
If one assumes moreover some a priori geometric estimates, we can get other sufficient or necessary conditions for the properties given in \tref{thm:trace}. The first set of conditions are the so called relative Faber-Krahn inequality that has been introduced by Grigor'yan in \cite{GriRMI}.
\begin{defi}We say that a complete  weighted Riemannian manifold $(M,g,\mu)$ satisfies the relative Faber-Krahn inequality if there are positive constants $\text{b},\upnu$ such that for any geodesic ball $B$ of radius $r(B)$ and any open subset $\cU\subset B\colon$
\begin{equation}\label{def:FK}\tag{$\text{FK}_{\text{b},\upnu}$}\frac{ \text{b}}{r^2(B)}\left(\frac{\mu(\cU)}{\mu(B)}\right)^{-\frac2\upnu}\le \lambda_1(\cU),\end{equation}
where $\lambda_1(\cU)$ is the lowest eigenvalue of the Dirichlet Laplacian on $\cU$.
\end{defi} 
According to \cite{GriRMI}, the relative Faber-Krahn inequality is equivalent to the conjonction of a doubling property \eqref{doubling} and of the upper Gaussian estimate for the heat kernel
\begin{equation}\label{GUE}\tag{GUE}
\forall x,y\in M,\ \forall t>0\colon P(t,x,y)\le \frac{C}{\mu\left(B(x,\sqrt{t})\right)}e^{-\frac{d^2(x,y)}{5t}}.
\end{equation}

Then using the appendix, we can deduce the following:
\begin{thm}\label{thm:traceFK}  Let $(M,g,\mu)$ be a complete non-parabolic  weighted Riemannian manifold  satisfying the relative Faber-Krahn inequality  \eqref{def:FK}. There is a constant $C(\text{b},\upnu)$ depending only on $\text{b},\upnu$ with the property that if $q$ is non negative locally integrable function such that for some positive $A$ and for any geodesic ball $B\subset M\colon$
$$\int_B \left[\int_0^{+\infty} \left(\fint_{B(x,r)}q\un_Bd\mu\right)\, dr\right]^2 \,d\mu(x)\le A \int_B qd\mu$$ then
$$\forall \varphi\in \cC^\infty_c(M)\colon \ \int_M q\varphi^2\, d\mu\le C(\text{b},\upnu)\, A\int_M |d\varphi|^2\,d\mu$$
\end{thm}
\begin{proof}[Hint on the proof of \tref{thm:traceFK}:]
According to \tref{thm:Genetrace} {\bf iv)$\Rightarrow$i)}, we get that if $\cK$ is the operator whose Schwartz kernel is given by
$$K(x,y)=\int_{d(x,y)}^{+\infty} \frac{dr}{\mu\left(B(x,r)\right)},$$
that is to say
$$(\cK f)(x)=\int_0^{+\infty} \left(\fint_{B(x,r)} fd\mu\right) dr$$ then the operator $\cK\sqrt{q}$ has a bounded extension to $L^2_\mu$ with 
$$\left\|\cK\sqrt{q}\right\|^2_{L^2_\mu\to L^2_\mu }\le C A.$$
\pref{prop:estikernel} implies that the Schwartz kernel of $\Delta^{-\frac12}$ is dominated by the Schwartz kernel of $\cK$:

$$\forall x\not=y\in M\colon \Delta^{-\frac12}(x,y)=\int_0^{+\infty} P(t,x,y)\frac{dt}{\sqrt{\pi \, t}}\le C\int_{d(x,y)}^{+\infty} \frac{dr}{\mu\left(B(x,r)\right)}=K(x,y).$$
Hence the operator $\Delta^{-1}\sqrt{q}$ has also a bounded extension to $L^2_\mu$ with 
$$\left\|\Delta^{-1}\sqrt{q}\right\|^2_{L^2_\mu\to L^2_\mu }\le C A.$$
and we conclude with the fact that 
$$\left\|\Delta^{-1}\sqrt{q}\right\|_{L^2_\mu\to L^2_\mu }= \left\|\sqrt{q}\Delta^{-1}\right\|_{L^2_\mu\to L^2_\mu }$$ and  with \tref{thm:trace}{\bf ii)$\Rightarrow$i)}.
\end{proof}
\subsection{Poincaré inequalities}
\subsubsection{}The second set of geometric conditions are the so-called doubling condition and the Poincaré inequalities (\ref{def:P}).

Recall that according to \cite{Griharnack,Saloffharnack,Salofflivre}, the conjonction of doubling condition \eqref{doubling} and of the Poincaré inequalities  \eqref{def:P} is equivalent to the upper and lower Gaussian estimate of the heat kernel:
\begin{equation}\label{DUE}\tag{DUE}
\forall x,y\in M,\ \forall t>0\colon  \frac{c}{\mu\left(B(x,\sqrt{t})\right)}e^{- \frac{d^2(x,y)}{ct}}\le P(t,x,y)\le \frac{C}{\mu\left(B(x,\sqrt{t})\right)}e^{-\frac{d^2(x,y)}{5t}}.
\end{equation}
Note that the lower bound on the heat kernel and the doubling condition \eqref{doubling} easily imply 
\begin{equation}\label{lowerPball}
\forall o\in M,\ \forall t>0\ \forall x,y\in B(o,\sqrt{t}) \colon \frac{c'}{\mu\left(B(o,\sqrt{t})\right)}\le P(t,x,y).\end{equation}
Hence the doubling condition together with  the Poincaré inequalities  imply the relative Faber-Krahn inequality but the reciprocal is not true for instance for $n\ge2$, the connected sum $\R^n\#\R^n$ satisfies the relative Faber-Krahn inequality but not the Poincaré inequalities  \eqref{def:P}.
\subsubsection{Harnack inequalities for harmonic functions}
We also know that the conditions \eqref{doubling}$+$\eqref{def:P} are also equivalent to Parabolic Harnack inequalities for positive solution of the heat equation  \cite{Griharnack,Saloffharnack,Salofflivre}, hence they imply the elliptic Harnack inequalities
\begin{prop}\label{prop:harnack0}If $(M,g,\mu)$ is a complete  weighted Riemannian manifold satisfying the doubling condition\eqref{doubling} and the Poincaré inequalities \eqref{def:P} then there is a constant $C$ depending only on $\upkappa,\upnu,\uplambda$ such that for any geodesic ball $B$ and any positive harmonic function $\psi\colon 2B\rightarrow \R_+^*$:
$$\sup_B \psi\le C\, \inf_B \psi.$$
\end{prop}
These Harnack inequalities imply more properties for harmonic functions:
\begin{prop}\label{prop:harnack}If $(M,g,\mu)$ is a complete  weighted Riemannian manifold satisfying the doubling condition \eqref{doubling} and of the Poincaré inequalities  \eqref{def:P}  then there are  constants $C>0, p_+>1,\alpha\in (0,1]$ depending only on $\upkappa,\upnu,\uplambda$ such that for any geodesic ball $B$ of radius $R$ and any  harmonic function $\psi\colon 3B\rightarrow \R_+^*$:
\begin{enumerate}[i)]
\item For any $y,z\in B\colon$
$$ \left|\psi(y)-\psi(z)\right|\le C\left(\frac{d(y,z)}{R}\right)^\alpha \sup_{x\in 2B} |\psi(x)|,$$
\item for any $0<\theta_1\le \theta_2\le 2\colon$
$$ \theta_1^{2-2\alpha}\fint_{\theta_1B} |d\psi|^2\, d\mu\le C \theta_2^{2-2\alpha}\fint_{\theta_2B} |d\psi|^2\, d\mu,$$
\item 
$\displaystyle \left(\fint_{B} |d\psi|^{p_+}\, d\mu\right)^{\frac 2p_+}\le C\, \fint_{2B} |d\psi|^2\, d\mu.$
\end{enumerate}
\end{prop}
The  H\"older regularity of harmonic function \pref{prop:harnack}-i), is a classical consequence of the Harnack inequality (see for instance \cite[Proof of Theorem 8.22]{GT}). The proof of \cite[Proposition 5.5]{caRMI} implies that  \pref{prop:harnack}-ii) holds under Poincaré inequalities and the doubling condition. The reverse  Hölder property for the gradient of harmonic functions \pref{prop:harnack}-iii)  is proven by Auscher and Coulhon \cite[Subsection 2.1]{ACPisa}.
\subsubsection{Another equivalent condition for \tref{thm:trace}}
Similarly to \tref{thm:traceFK}, one obtains:
\begin{thm}\label{thm:traceDP} Let $(M,g,\mu)$ be a complete non-parabolic weighted Riemannian manifold satisfying the doubling condition \eqref{doubling} and of the Poincaré inequalities  \eqref{def:P}  and $q$ be a non negative, locally integrable function then the following properties are equivalents
\begin{enumerate}[i)]
\item There is a constant $C_1$ such that \eqref{trace1} holds.
\item $Q=\Delta^{-\frac12} q$ is finite a.e. and there is a constant $C_{ii}$ such that
$$\Delta^{-\frac12} (Q^2 )\le C_{ii} \, Q$$
\item there is a constant $C_{iii}$ such that for any geodesic ball $B$:
$$\int_B \left| \Delta^{-\frac12}\left(q\un_B\,\right)\right|^2(x)d\mu(x)\le C_{iii}\int_B q\,d\mu$$
\item $\tilde Q(x)=\int_0^{+\infty} \left(\fint_{B(x,r)}q\,d\mu\right)dr$ is finite almost a.e. and there is a constant $C_{iv}$ such that 
$$\int_0^{+\infty}  \left(\fint_{B(x,r)}\tilde Q^2\,d\mu\right)dr\le C_{iv} \tilde Q(x).$$
\end{enumerate}
Moreover the constants $C_1, C_{ii},  C_{iii},  C_{iv}$ are mutually controlled.

\end{thm}
\begin{proof}[Hint on the proof of \tref{thm:traceDP}:]
We introduce again the operator $\cK$ whose Schwartz kernel is given by
$$K(x,y)=\int_{d(x,y)}^{+\infty} \frac{dr}{\mu\left(B(x,r)\right)}.$$
Then using the lower and upper Gaussian estimate for the heat kernel \eqref{DUE}, \pref{prop:estikernel} yields that
$$\forall x\not=y\in M\colon c\, K(x,y)\le \Delta^{-\frac12}(x,y)\le C\,K(x,y).$$
So that  if \eqref{trace1} holds with constant $C_1$ then the operator $\cK\sqrt{q}$ has a bounded extension to $L^2_\mu$ and $\left\|\cK\sqrt{q}\right\|^2_{L^2_\mu\to L^2_\mu }\le C\, C_1.$
And if the operator $\cK\sqrt{q}$ has a bounded extension to $L^2_\mu$  then \eqref{trace1} holds with constant $C_1=c^{-1} \left\|\cK\sqrt{q}\right\|^2_{L^2_\mu\to L^2_\mu }$.

Hence \tref{thm:Genetrace} {\bf iv)$\Leftrightarrow$i)} implies the equivalence {\bf iv)$\Leftrightarrow$i)} in \tref{thm:traceDP}. 
The estimates
$$c\,\tilde Q\le Q\le C\,\tilde Q\text{ and }c\,\cK(\tilde Q^2)\le \Delta^{-\frac12}Q^2\le C\, \cK(\tilde Q^2)$$
imply the equivalence {\bf iv)$\Leftrightarrow$ii)} in \tref{thm:traceDP}. 
Similarly the equivalence {\bf iii)$\Leftrightarrow$iv)} in \tref{thm:traceDP} is a consequence of \tref{thm:Genetrace} {\bf iii)$\Leftrightarrow$iv)}.

\end{proof}
\section{A stronger $L^2$-weighted boundedness property for the Hodge projector}
\subsection{Muckenhoupt weight properties of equilibrium potential}
We recall:
\begin{defi}\label{def:A1}
If $(X,d,\mu)$ is a measure metric space satisfying the doubling condition \eqref{doubling1}, then a positive locally integrable function $\omega$ is said to be a $\Aun\!\!-$weight with constant $C$ if for any ball $B\subset X\colon$
$$ \fint_B \omega\,d\mu\le C\,\inf_B\omega.$$
\end{defi}
The following properties of $\Aun\!\!-$weight are classical \cite{Korte,Poidhardy}:
\begin{prop}\label{prop=A1}Let $\omega$ be a $\Aun\!\!-$weight with constant $C$ on some measure metric space $(X,d,\mu)$ satisfying the doubling condition \eqref{doubling1}. Then there is a constant $D$ depending only on $\upkappa,\upnu$ and $C$ such that for any ball $B\subset X\colon$
\begin{enumerate}[i)]
\item $\int_{2B} \omega d\mu\le D\int_{B} \omega d\mu$, 
\item  $\inf_B\omega\le \fint_B \omega\,d\mu\le C\,\inf_B\omega,$
\item $\inf_B\omega\le D\,\inf_{2B}\omega.$
\end{enumerate}
\end{prop}
And we also have the following useful criteria for a weight to satisfy a reverse Hölder inequality
\begin{cor}\label{cor:A2RH}  If $(X,d,\mu)$ is a measure metric space satisfying the doubling condition \eqref{doubling1} and if for some $r>1$, $\omega^r$ is a $\Aun\!\!-$weight with constant $C$ then
$\omega$ is a reverse Hölder weight meaning that for any ball $B\subset X\colon$
$$\left(\fint_B \omega^rd\mu\right)^{\frac1r}\le C^{1/r}\,\fint_B \omega\,d\mu.$$
\end{cor} 
Our result concerning the weight properties of equilibrium potential is the following generalization of \cite[Lemma 2.1]{MVarkiv}:
 \begin{prop}\label{prop:hA1} Let $(M,g,\mu)$ be a complete non-parabolic weighted Riemannian manifold  satisfying the doubling condition \eqref{doubling} and the Poincaré inequalities  \eqref{def:P}. Then for any $\displaystyle\tau\in \left[0,\upnu/(\upnu-2)\right)$, there is a positive constant  $C$ depending only on $\upkappa,\upnu,\uplambda$ and $\tau$ such that for any bounded open set, its equilibrium potential $h$ satisfies that $h^\tau$ is a $\Aun\!\!-$weight with constant $C$:
$$\forall B\subset M\colon \fint_B h^\tau\,d\mu\le C\,\inf_B h^\tau.$$
\end{prop}
Note that the non-parabolicity condition forces that $\nu>2$ (remark \ref{nu2}). 
This proposition will be consequence of the same properties but only for Green functions.
\begin{lem}\label{lem:GreenA1} Under the assumption of \pref{prop:hA1}, for any $\displaystyle\tau\in \left[0,\upnu/(\upnu-2)\right)$, there is a positive constant $C$ depending only on $\upkappa,\upnu,\uplambda$ and $\tau$ such that for any $o\in M$ the Green function with pole at $o$ satisfies that for any ball $B\subset M\colon$
$$\fint_B G_o^\tau d\mu\le C\inf_B G_o^\tau.$$
\end{lem}
\begin{proof}[Proof of \pref{prop:hA1} assuming \lref{lem:GreenA1}]
It is enough to show the result when $\tau>1$. Let $\cU$ be a bounded open subset with equilibrium potential $h$ and equilibrium measure $\nu_\cU$:
$$h(x)=\int_M G(x,y)d\nu_\cU(y).$$
Let $B$ be a geodesic ball and let $f\in L^{\tau^*}(B)$ with $\frac1\tau+\frac1{\tau^*}=1$ and $\fint_B |f|^{\tau^*}d\mu=1$.
Then \begin{align*}
\left|\fint_B fhd\mu\right|&=\left|\int_M \left(\fint_B G(x,y)f(x)d\mu(x)\right)d\nu_\cU(y)\right|\\
&\le \int_M \left(\fint_B G^\tau(x,y)d\mu(x)\right)^{\frac 1\tau}d\nu_\cU(y)\\
&\le C^{\frac 1\tau}\int_M \left(\inf_{x\in B} G(x,y)\right)d\nu_\cU(y)\\
&\le C^{\frac 1\tau}\inf_B h.
\end{align*}
\end{proof}
\begin{proof}[Proof of \lref{lem:GreenA1}] The estimate \eqref{DUE} and \pref{prop:estikernel} imply the following estimate of the Green kernel:
\begin{equation}\label{estiGreen}c\int_{d(x,y)}^{+\infty}\frac{r}{\mu\left(B(x,r)\right)}dr\le G(x,y)\le C\int_{d(x,y)}^{+\infty}\frac{r}{\mu\left(B(x,r)\right)}dr.\end{equation}
Notice that using the doubling condition one easily gets that 
$$\int_{s}^{+\infty}\frac{r}{\mu\left(B(x,r)\right)}dr\ge \int_{s}^{2s}\frac{r}{\mu\left(B(x,r)\right)}dr\ge \frac{s^2}{\upgamma \mu(B(x,s))},$$
so that for any $x,y\in M:$
\begin{equation}\label{estiGreen3}
G(x,y)\ge c \frac{d^2(x,y)}{\mu\left( B(x,d(x,y))\right)}.
\end{equation}
Moreover if $y\in B(x,r)$ then $d(x,y)\le r$ and 
$$G(x,y)\ge c \int_{r}^{+\infty}\frac{r}{\mu\left(B(x,r)\right)}dr\ge  c\frac{r^2}{\upgamma \mu(B(x,r))},$$
hence we deduce that 
\begin{equation}\label{estiGreen0}\inf_{y\in B(x,r)} G(x,y)\ge c \int_{r}^{+\infty}\frac{r}{\mu\left(B(x,r)\right)}dr\ge  c\frac{r^2}{\upgamma \mu(B(x,r))},\end{equation}
Let $\tau\in \left[1,\upnu/(\upnu-2)\right)$ and $o\in M$ and $B=B(z,r)$. There are 3 cases to be considered
\\
\noindent{\bf First case:}  {$\bf r\le d(o,z)/4$}. In that case $x\mapsto G_o(x)$ is a positive harmonic function on $3B$ and the desired conclusion is then a direct consequence of the Harnack inequality
$$\sup_B G_o\le C \inf_B G_o.$$ \\
\noindent{\bf Second case: } ${\bf o=z}$. Let $f\in L^{\tau^*}(B)$ with $\frac1\tau+\frac1{\tau^*}=1$ and $\fint_B |f|^{\tau^*}d\mu=1$ then
$$\left|\fint_B fG_od\mu\right|\le C\fint_B\left( \int_{d(o,y)}^{+\infty}\frac{s}{\mu\left(B(o,s)\right)}ds\right)|f|(y)d\mu(y)=C\left(\textbf{I}+\textbf{II}\right)$$
with
$$\textbf{I}=\left(\int_{r}^{+\infty}\frac{s}{\mu\left(B(o,r)\right)}ds\right)\fint_B|f|(y)d\mu(y)$$
$$\text{ and } \textbf{II}=\fint_B \, \left(\int_{d(o,y)}^{r}\frac{s}{\mu\left(B(o,r)\right)}ds\right)|f|(y)d\mu(y).$$
Using that $\fint_B |f|^{\tau^*}d\mu=1$ and \eqref{estiGreen0}
one gets that 
$$\textbf{I}\le \frac{1}{c} \inf_B G_o.$$
For the second term:
\begin{align*}
\textbf{II}&= \fint_B\left( \int_{d(o,y)}^{r}\frac{s}{\mu\left(B(o,s)\right)}ds\right)|f|(y)d\mu(y).\\
&=\frac{1}{\mu(B)}\int_0^rs\left(\fint_{B(o,s)} |f|(y)dy\right) ds\\
&\le \frac{1}{\mu(B)}\int_0^rs\left(\frac{\mu\left(B(o,r)\right)}{\mu\left(B(o,s)\right)}\right)^{\frac1{\tau^*}}ds,
\end{align*}
where in the last line we used Hölder inequality and $\fint_B |f|^{\tau^*}d\mu=1$.
Using the doubling condition \eqref{doubling} one gets
$$\int_0^rs\left(\frac{\mu\left(B(o,r)\right)}{\mu\left(B(o,s)\right)}\right)^{\frac1{\tau^*}}ds\le \upkappa^{\frac{1}{\tau^*}} \int_0^rs\left(\frac{r}{s}\right)^{\frac\upnu {\tau^*}}ds=\frac{\upkappa^{\frac1{\tau^*}}\,\tau}{\upnu-(\upnu-2)\tau}\, r^2.$$
Hence we deduce that 
$$\left(\fint_B G_o^\tau d\mu\right)^{\frac 1\tau}\le \frac{C}{c} \inf_B G_o+\frac{C \upkappa^{\frac{1}{\tau^*}}\,\tau}{\upnu-(\upnu-2)\tau}\frac{r^2}{\mu(B(o,r))}.$$
As \eqref{estiGreen0} yields that 
$\frac{r^2}{\mu(B(o,r))}\le C \inf_{B} G_o$, we get
$$\left(\fint_B G_o^\tau d\mu\right)^{\frac 1\tau}\le C' \inf_B G_o.$$

\noindent{\bf Third case: }{$\bf 4r>d(o,z)>0$. } In that case, the result follows from $B(z,r)\subset B(o,5r)$: \begin{align*}
\fint_B G_o^\tau d\mu&\le \upgamma^3\fint_{B(o,5r)}G_o^\tau d\mu\\
&\le  \upgamma^3C\inf_{B(o,5r)}G_o^\tau \text{  (using the result obtained in the second case)}\\
&\le  \upgamma^3C\inf_{B}G_o^\tau.
\end{align*}

\end{proof}

\subsection{Reverse Hölder for gradient of harmonic functions}
\begin{lem}\label{lem:GHolder}  Let $(M,g,\mu)$ be a complete non-parabolic weighted Riemannian manifold  satisfying the doubling condition \eqref{doubling} and the Poincaré inequalities \eqref{def:P}. There is some $q>2$ and some positive constant $C$ that depend only on $\upkappa,\upnu,\uplambda$ such that for any $o\in M$ and any ball $B\subset M$ and any harmonic function $\psi\colon 3B\rightarrow \R$ satisfies the following reverse Hölder property:
$$\fint_B G_o(y) |d\psi|^{q}d\mu\le C \inf_B G_o\, \left(\fint_{2B}  |d\psi|^2d\mu\right)^{\frac{q}2}.$$
\end{lem}
\proof Proceeding as in the proof of \lref{lem:GreenA1} and provided ${q}\le p_+$ where $p_+$ is given \pref{prop:harnack}, it is enough to prove the result for balls centered at $o$. So assume that $B=B(o,r)$.
Let ${q}>2$ be such that for the constants $\alpha\in (0,1]$ and $p_+>2$ given by \pref{prop:harnack}:
$${q}\le p_+\text{ and } {q}(1-\alpha)<2.$$
Using the Green kernel estimate \eqref{estiGreen} and proceeding as before, we get for any harmonic function $\psi\colon 3B\rightarrow \R$:
\begin{equation}\label{first1}\fint_B G_o(y) |d\psi|^{q}d\mu\le C\inf_B G_o(y)\fint_{B}  |d\psi|^{q}d\mu+C \textbf{II}\end{equation}
with \begin{align*}
\mu(B)\textbf{II}&= \int_0^r s\left(\fint_{B(o,s)} |d\psi|^{q}d\mu\right) ds\\
&\le  \int_0^r s\left(\fint_{B(o,2s)} |d\psi|^2d\mu\right)^{\frac{q}2} ds \text{  using \pref{prop:harnack}.iii)}\\
&\le  \int_0^r s\left(\frac{r}{s}\right)^{{q}(1-\alpha)} ds\left(\fint_{B(o,2r)} |d\psi|^2d\mu\right)^{\frac{q}2}\text{  using \pref{prop:harnack}.ii)}\end{align*}
Hence \begin{equation}\label{eq:second}
\textbf{II}\le  C \frac{r^2}{\mu(B)}\left(\fint_{B(o,2r)} |d\psi|^2d\mu\right)^{\frac{q}2}\le C \inf_B G_o  \left(\fint_{B(o,2r)} |d\psi|^2d\mu\right)^{\frac{q}2},
\end{equation}
where we have used  \eqref{estiGreen0}:  $$\frac{r^2}{\mu(B)}\le C \inf_B G_o.$$
Hence the inequality \eqref{first1}, the reverse Hölder inequality
$$\fint_B  |d\psi|^{q}d\mu\le C  \left(\fint_{2B}  |d\psi|^2d\mu\right)^{\frac{q}2}$$ 
and \eqref{eq:second}
yields that \begin{equation}\label{eq:premier}\fint_B G_o(y) |d\psi|^{q}d\mu\le C\inf_B G_o(y)\, \left(\fint_{2B}  |d\psi|^2d\mu\right)^{\frac{q}2}.\end{equation}

\endproof
This lemma has the following crucial consequence:\begin{prop}\label{prop:RH} Let $(M,g,\mu)$ be a complete non-parabolic weighted Riemannian manifold  satisfying the doubling condition \eqref{doubling} and the Poincaré inequalities \eqref{def:P}. There is some ${q}>2$ and some positive constant $C$ that depend only on $\upkappa,\upnu,\uplambda$ such that for any equilibrium potential $h\colon M\rightarrow \R$, any $\tau\in [0,1]$, any ball $B\subset M$ and any harmonic function $\psi\colon 3B\rightarrow \R$
$$\frac{1}{\fint_B h^\tau}\fint_B h^\tau |d\psi|^{q}d\mu\le C  \left(\frac{1}{\fint_{2B} h^\tau}\fint_{2B}  h^\tau |d\psi|^2d\mu\right)^{\frac{q}2}.$$
\end{prop}
\proof Let $h\colon M\rightarrow \R$ be the equilibrium potential of some bounded open set $\cU\subset M$, $\tau\in (0,1]$, $B\subset M$ a geodesic ball and let $\psi\colon 3B\rightarrow \R$ be some harmonic function.
Using Hölder inequality, we get
$$\fint_B h^\tau |d\psi|^{q}d\mu\le\left(\fint_B h |d\psi|^{q}d\mu\right)^\tau\left(\fint_B  |d\psi|^{q}d\mu\right)^{1-\tau}.$$
If $\nu_\cU$ is the equilibrium measure associated to $\cU$, we have
\begin{align*}
\fint_B h |d\psi|^{q}d\mu&=\int_M \left(\fint_B G(x,y)|d\psi|^{q}(y)d\mu(y)\right)d\nu_\cU(x)\\
&\le C\int_M \inf_{y\in B} G(x,y)d\nu_\cU(x)\left(\fint_{2B} |d\psi|^2d\mu\right)^{\frac{q}2}\\
&\le C \inf_{y\in B} h(y)\,\left(\fint_{2B} |d\psi|^2d\mu\right)^{\frac{q}2}.
\end{align*}
So that 
\begin{align*}\fint_B h^\tau |d\psi|^{q}d\mu&\le C \inf_{x\in B} h^\tau(x)\,\left(\fint_{2B} |d\psi|^2d\mu\right)^{\frac{q}2}\\
&\le C \left(\inf_{x\in B} h^\tau(x)\right) \left(\inf_{x\in 2B} h^\tau(x)\right)^{-\frac{q}2}\,\left(\fint_{2B} h^\tau |d\psi|^2d\mu\right)^{\frac{q}2}.\end{align*}

Using the fact that $h$ is a $\Aun\!\!-$weight and \pref{prop=A1}, we get that
$$\inf_{x\in B} h^\tau(x)\le \fint_B h^\tau d\mu\text{ and }  \fint_{2B} h^\tau d\mu\le C \inf_{x\in 2B} h^\tau(x).$$
Hence the result.
\endproof
\subsection{$L^2$-weighted boundedness of the Hodge projector}
\begin{thm}
\label{thm:WHodge} Let $(M,g,\mu)$ be a complete non-parabolic weighted Riemannian manifold  satisfying the doubling condition \eqref{doubling} and the Poincaré inequalities \eqref{def:P}. There are constants $\tau_+>1$ and   $C>0$  both depending only on $\upkappa,\upnu,\uplambda$ such that for any equilibrium potential $h\colon M\rightarrow \R$,
\begin{equation}\label{Hodgheimprov}
\int_M \left|\Pi(\beta)\right|^2\, h^{\tau_+} d\mu\le C \int_M \left|\beta\right|^2\, h^{\tau_+} d\mu.\end{equation}
\end{thm}
\proof
Let $h\colon M\rightarrow \R$ be some equilibrium potential.

According to Auscher and Coulhon \cite{ACPisa}, there are  $p_-<2$ and a constant $C$ both depending only on $\upkappa,\upnu,\uplambda$ such that for any $r\in [p_-,p_-^*]$, the Hodge projector has a bounded extension to $L^r(T^*M)$ with 
$$\left\|\Pi\right\|_{L^r\to L^r}\le C.$$
Let $\delta:=\frac{q-2}{q}\frac{1}{\upnu-2}$ where $q$ is given by  \pref{prop:RH} and 
let $p<2$ be given by
$$\frac{1}{p}=\frac{\sqrt{1-\delta/2}}{2}+\frac{1-\sqrt{1-\delta/2}}{p_-}$$
and $\sigma\in (0,1)$ given by
$$\sigma=1-\frac{\delta}{2}=\left(\sqrt{1-\frac{\delta}{2}}\right)^2.$$ 
 Interpolating \cite[Theorem 2]{Stein_1956} between the boundedness
$$\left\|\Pi\right\|_{L^{p_-}\to L^{p_-}}\le C$$ and the weighted $L^2$ boundedness
given by \tref{thm:L2Whodge}
$$\forall \beta \in \cC^\infty_c(T^*M)\colon
\int_M \left|\Pi(\beta)\right|^2 \, h^{\sqrt{\sigma}} d\mu\le \left(\frac{1+\sqrt{\sigma} }{1- \sqrt{\sigma} }\right)^2\  \int_M \left|\beta\right|^2 \, h^{\sqrt{\sigma}}  d\mu,$$
we deduce the weighted $L^p$ boundedness
$$\forall \beta \in \cC^\infty_c(T^*M)\colon
\int_M \left|\Pi(\beta)\right|^p \, h^{\sigma} d\mu\le C_1 \  \int_M \left|\beta\right|^p \, h^{\sigma}  d\mu.$$
Notice that \pref{prop:RH} yields that for any ball $B\subset M$ and any $\beta\in L^2(T^*M)$ with support in $M\setminus 3B$ then
$$\left(\frac{1}{\int_{B} h^{\sigma} d\mu}\int_{B} \left|\Pi(\beta)\right|^q \, h^{\sigma} d\mu\right)^{\frac 1q}\le C_2\left(\frac{1}{\int_{2B} h^{\sigma} d\mu}\int_{2B} \left|\Pi(\beta)\right|^p \, h^{\sigma} d\mu\right)^{\frac 1p}.$$
The measure $h^\sigma d\mu$ is doubling hence we can use  a result of Auscher and Martell \cite[Theorem 3.14]{AMadv} and we know that if $\omega$ is a positive locally integrable weight satisfying for some positive constant $D$:
\begin{itemize}
\item  $\omega$ is a $A_2-$weight on $(M,d_g,h^\sigma d\mu)$: for any ball $B\subset M$ then
$$\left(\frac{1}{\int_{B} h^{\sigma} }\int_B \omega \,h^\sigma d\mu\right)\left(\frac{1}{\int_{B} h^{\sigma} }\int_B \omega^{-1} h^\sigma d\mu\right)\le D,$$ 
\item  $\omega$ is a reverse Hölder weight of exponent $(q/2)^*=\frac{q}{q-2}$ on $(M,d_g,h^\sigma d\mu)$: 
$$ \left(\frac{1}{\int_B h^\sigma}\int_B \omega^{\frac{q}{q-2}} h^\sigma d\mu\right)^{1-\frac 2q}\le D \frac{1}{\int_B h^\sigma}\int_B \omega\, h^\sigma d\mu,$$
\end{itemize}
then for a constant $C$ depending only on $\upkappa,\upnu,C_1, C_2$ and $D$, we get the weighted $L^2$ boundedness of the Hodge projector:
$$\forall \beta \in \cC^\infty_c(T^*M)\colon
\int_M \left|\Pi(\beta)\right|^2 \, \omega h^{\sigma} d\mu\le C \  \int_M \left|\beta\right|^2 \, \omega h^{\sigma}  d\mu.$$
From \cref{cor:A2RH}, we know that  $\omega=h^a$ satisfies these conditions when $h^{\sigma+\frac{q}{q-2}a}$ is a 
$\Aun\!\!-$ weight on $(M,d_g, \mu)$, hence it is the case when $$\sigma+\frac{q}{q-2}a\le \frac{\upnu}{\upnu-2}.$$
Considering $\omega=h^{\delta}$, then $$h^\sigma\omega^\frac{q}{q-2}=h^{1+\left(\frac12+\frac1q\right)\frac{1}{\upnu-2}}$$ is a $\Aun\!\!-$ weight on $(M,d_g, d\mu)$.Hence the result of Auscher and Martell implies the uniform  weighted $L^2$ boundedness of the Hodge projector \eqref{Hodgheimprov}  for $\tau_+=1+\delta/2$.
\endproof

\section{The Maz'ya-Vertbisky problem}
In this section we will prove \tref{thm:MV}.
\subsection{The direct implication} We start by the easiest implication and according to Maz'ya and Vertbisky, this idea comes for \cite[Lemma2.1]{Combescure1976}. 

\begin{prop} Let $(M,g,\mu)$ be a complete non-parabolic weighted Riemannian manifold and $\theta\in L^2_{loc}(T^*M)$ such that for some positive constant $\textbf{A}$:
$$\forall \varphi\in \cC^\infty_c(M)\colon \int_M |\theta|^2\varphi^2\,d\mu\le \,\textbf{A}^2\, \int_M |d\varphi|^2\,d\mu$$
then the distribution $V=d^*_\mu\theta$ satisfies
$$\forall \varphi,\phi\in \cC^\infty_c(M)\colon\left|\la V, \varphi\phi\ra\right|\le 2\textbf{A} \|d\varphi\|_{L^2_\mu}\,\|d\phi\|_{L^2_\mu}.$$
\end{prop}
\proof  By definition:
$$\la V, \varphi\phi\ra=\int_M\la \theta,d(\varphi\phi)\ra\,d\mu=\int_M\la \theta,\phi d\varphi\ra\,d\mu+\int_M\la \theta,\varphi d\phi\ra\,d\mu.$$ And using the Cauchy-Schwartz inequality and the hypothesis, one gets:
\begin{align*}\left|\la V, \varphi\phi\ra\right|&\le \| \phi\theta\|_{L^2_\mu}\, \|d\varphi\|_{L^2_\mu}+\| \varphi\theta\|_{L^2_\mu}\, \|d\phi\|_{L^2_\mu}\\
&\le \textbf{A} \|d\phi\|_{L^2_\mu}\|d\varphi\|_{L^2_\mu}+\textbf{A} \|d\varphi\|_{L^2_\mu}\|d\phi\|_{L^2_\mu}. \end{align*}
\endproof
\subsection{} In order to be able to prove a reciprocal, we start by the following notion
\begin{defi}\label{def:wpara} If $(M,g,\mu)$ is a complete weighted Riemannian manifold, a non negative, locally integrable function $\omega$ is said to be a {\it parabolic} weight if there is a sequence $\chi_\ell\in \cC^\infty_c(M)$ such that 
\begin{equation}\label{paracutoff}\begin{cases}
0\le \chi_\ell\le  1&\text{ every where on }M\\
\lim_{\ell\to+\infty} \chi_\ell= 1&\text{  uniformly on compact set }\\
\lim_{\ell\to+\infty}\int_M |d\chi_\ell|^2\,\omega d\mu=0&\end{cases}\end{equation}
\end{defi}
\begin{rem}\label{rem:compar}
When $\omega$ is positive, this is equivalent to the parabolicity of the weighted Riemannian manifold $(M,g,\omega\mu)$. It is clear that if $\omega_1\le \omega_2$, then the parabolicity of $\omega_2$ implies the parabolicity of $\omega_1$. 
\end{rem}
Our first result is a refinement of \cite[Proposition 2.27]{carkato}:
\begin{lem}\label{lem:wpara} If $(M,g,\mu)$ is  a complete non-parabolic weighted Riemannian manifold satisfying the doubling condition \eqref{doubling} and the Poincaré inequalities $\eqref{def:P}$ then for any $o\in M$, $\min\{ G_o,1\}$ is a parabolic weight.
\end{lem}
\proof Recall \ref{estiGreen}:
\begin{equation}\label{restiGreen}
C^{-1} \int_{d(x,o)}^{+\infty} \frac{r}{\mu(B(o,r))} dr\le G_o(x)\le C \int_{d(x,o)}^{+\infty} \frac{r}{\mu(B(o,r))} dr,\end{equation}
hence the non parabolicity implies  the finiteness of integral :
$$\int_{1}^{+\infty} \frac{r}{\mu(B(o,r))} dr<+\infty$$ and  that 
$$\lim_{x\to+\infty} G_o(x)=0.$$ 
We let $u\colon\R\rightarrow [0,1]$ be a smooth function such that 
$$\begin{cases} u(t)=0&\text{ if } t\le 0\\
u(t)=1&\text{ if } t\ge 1\\
|u'(t)|\le 2&\text{ everywhere}.\end{cases}$$ and define
$$\chi_\ell(x)=u\left(\frac{\log (G_o(x) \ell^2)}{\log(\ell)}\right).$$ The estimates \eqref{restiGreen} implies that 
$\chi_\ell \cC_c^\infty(M)$ and
$$\lim_{\ell\to+\infty} \chi_\ell= 1\text{  uniformly on compact set}.$$
We have the following estimations
\begin{align*}
\int_M |d\chi_\ell|^2\,\min\{ G_o,1\} d\mu&\le4\log(\ell)^{-2}\int_{\{\ell^{-2}<G_o<\ell^{-1}\}} \frac{|dG_o|^2}{G_o}d\mu\\
&\le 4\log(\ell)^{-2}\int_{1/\ell^2}^{1/\ell} \left(\int_{\{G_o=x\}} |dG_o|\right)\frac{dx}{x}\\
&\le 4\log(\ell)^{-1},
\end{align*}
where we have used that for almost every $x$
$$\int_{\{G_o=x\}} |dG_o|=1.$$
Hence $\min\{ G_o,1\}$ is  a parabolic weight.
\endproof

We are now in position to finish the proof of our main result.
\begin{proof}[Proof of the reciprocal in \tref{thm:MV}:]
Let  $(M,g,\mu)$ be a complete non-parabolic weighted Riemannian manifold satisfying the the doubling condition \eqref{doubling} and the Poincaré inequalities $\eqref{def:P}$. And let $V$ be a distribution such that 
$$\forall \varphi,\phi\in \cC^\infty_c(M)\colon\left|\la V, \varphi\phi\ra\right|\le \textbf{A} \|d\varphi\|_{L^2_\mu}\,\|d\phi\|_{L^2_\mu}.$$
We want to define a distribution $\theta\in \cC^{-\infty}(T^*M)$ by the following relation: let $\beta\in \cC^\infty_c(T^*M)$ and define $\varphi_\beta\in \hW(M)$ to be  the solution of the equation
$$\Delta\varphi_\beta=d^*_\mu\beta$$ so that $\Pi(\beta)=d\varphi_\beta$ and let
$$\la \theta,\beta\ra=\la V,\varphi_\beta\ra.$$
We need to justify that such definition makes sense.

 Let $\cU$ be a bounded open subset containing the support of $\beta$ and let $h$ be its equilibrium potential.
Recall that by \pref{prop:harnack}-i), there is some $\alpha\in (0,1]$ such that we get the uniform  Hölder regularity for harmonic functions. \\
\noindent\textbf{ Claim: } if $\tau\in [1,1+2\alpha/(\upnu-2)]$,  then $h^{-\tau}\varphi^2_\beta$ is a parabolic weight.
\begin{proof}[Proof of the Claim] Let $o\in \cU$ and $R>0$ such that $\cU\subset B(o,R)$. According to the maximum principle, we have that on $M\setminus \cU$:
\begin{equation}
\label{estihG}\frac{1}{\max_{\partial \cU}G_o} G_o\le h\le \frac{1}{\min_{\partial \cU}G_o} G_o.\end{equation}
Moreover as $\int_M d_\mu^*\beta d\mu=0$, we have 
$$\varphi_\beta(x)=\int_{B(o,R)}\left(G(x,y)-G(x,o)\right) d^*_\mu\beta(y)d\mu(y).$$

Now if $x\not\in B(o,6R)$ then $y\mapsto G(x,y)$ is a harmonic function on $B(o,d(o,x)/2)$ hence \pref{prop:harnack} implies that for any $y\in B(o,R)\colon$
\begin{align}\
\left| G(x,y)-G(x,o)\right|&\le C\left(\frac{R}{d(o,x)}\right)^\alpha\sup_{z\in B(o,d(o,x)/3)} G(x,z)\notag\\
&\le C\left(\frac{R}{d(o,x)}\right)^\alpha G(x,o),\label{eq:est1}
\end{align}
where in the last inequality, we used the Harnack inequality given by \pref{prop:harnack0}  for the positive harmonic function $y\in B(o,2d(o,x)/3)\mapsto G(x,y)$.

Using \eqref{estiGreen3} and the doubling hypothesis, we also know that if  $x\not\in B(o,6R)$ then 
\begin{equation}\label{eq:est2}G(o,x)\ge C \frac{d^2(o,x)}{\mu\left(B(o,d(o,x))\right)}\ge C \frac{d^2(o,x)}{\mu\left(B(o,R\right)}\left(\frac{R}{d(o,x)}\right)^\upnu=C \frac{R^\upnu}{\mu\left(B(o,R)\right)}\frac{1}{d^{\upnu-2}(o,x)}.\end{equation}
The estimates \eqref{eq:est1} and  \eqref{eq:est2} yields that outside $B(o,6R)\colon$
$$|\varphi_\beta(x)|\le  C \left(\frac{\mu\left(B(o,R)\right)}{R^2}\right)^{\frac{\alpha}{\upnu-2}}G^{1+\frac{\alpha}{\upnu-2}}_o(x)$$
and estimates \eqref{estihG} and \lref{lem:wpara} implies that if $\tau\in [1,1+2\alpha/(\upnu-2)]$,  then $h^{-\tau}\varphi^2_\beta$ is bounded by a parabolic weight, hence it is a  parabolic weight.\end{proof}

Now let $\tau>1$ be such that $\tau\le{\tau_+}$ and $\tau\le 1+\frac{2\alpha}{\upnu-2},$ where $\tau_+$ is given by \tref{thm:WHodge}. By interpolation,  the Hodge projector extends also continuously on $L^2(T^*M, h^{-\tau}d\mu)$, hence $h^{-\tau/2}d\varphi_\beta\in L^2$ with
\begin{equation}\label{estL2}
\int_M h^{-\tau} |d\varphi_\beta|^2d\mu\le C \int_{\cU} |\beta|^2\, d\mu,\end{equation}
where we use  that $h=1$ on $\supp\beta$.
The parabolicity of the weight $h^{-\tau}\varphi_\beta^2$ yields a sequence of good cut-off function $\chi_\ell\in \cC^\infty_c(M)$ such that 
$0\le \chi_\ell\le  1$, $(\chi_\ell)$ converges to the function $1$  uniformly on compact sets of $M$ and 
$$\lim_{\ell\to+\infty}\int_M |d\chi_\ell|^2\,h^{-\tau}\varphi_\beta^2 d\mu=0.$$
Using the Hardy inequality \eqref{hardyh}, we get that 
\begin{align*}
\frac{(1+\tau)^2}{4}\int_M \frac{|dh|^2}{h^2} \chi_\ell^2\varphi_\beta^2\,h^{-\tau} d\mu&\le \int_M \left|d\left(\chi_\ell \varphi_\beta\right)\right|^2\,h^{-\tau} d\mu\\
&\le 2\int_M |d\chi_\ell|^2\, \varphi_\beta^2h^{-\tau} d\mu+\int_M \chi^2_\ell |d\varphi_\beta|^2\,h^{-\tau} d\mu.
\end{align*}
Passing to the limit $\ell\to +\infty$ and using \eqref{estL2}, we get that 
\begin{equation}\label{estL2b}\int_M \frac{|dh|^2}{h^2} \varphi_\beta^2\,h^{-\tau} d\mu
\le \frac{4}{(1+\tau)^2}C \int_{\cU} |\beta|^2\, d\mu
\end{equation}
With
$$d(h^{-\tau/2}\varphi_\beta)=h^{-\tau/2}d\varphi_\beta-\frac{\tau}{2}\,h^{-\tau/2}\varphi_\beta\frac{dh}{h},$$ we get that $d(h^{-\tau/2}\varphi_\beta)\in L^2_\mu$ with
\begin{equation}\label{estL2c}
\int_M \left|d(h^{-\tau/2}\varphi_\beta)\right|^2d\mu\le C \int_{\cU} |\beta|^2\, d\mu\end{equation}
where $C$ depends only on the doubling and Poincaré constants.
Using again the sequence of the cut-off function $\chi_\ell$, we easily obtain that 
$h^{-\tau/2}\varphi_\beta\in \hW(M)$ and it is the $\hW(M)$-limit of the sequence 
$\left( \chi_\ell h^{-\tau/2}\varphi_\beta\right)_\ell$.
Noticed that the assumptions

$$\forall \varphi,\phi\in \cC^\infty_c(M)\colon\left|\la V, \varphi\phi\ra\right|\le \textbf{A} \|d\varphi\|_{L^2_\mu}\,\|d\phi\|_{L^2_\mu},$$
implies that the bilinear form $$(\varphi,\phi)\in \cC^\infty_c(M)\times \cC^\infty_c(M)\mapsto \la V, \varphi\phi\ra$$ extends continuously to a bilinear form on $\hW(M)\times \hW(M)$. 
As $h^{\tau/2}$ and $h^{-\tau/2}\varphi_\beta$ are both in $\hW(M)$, we obtain
that 
\begin{align*}
\left|\la \theta,\beta\ra\right|&=\left|\la V\,,\, \varphi_\beta\ra\right|\\
&=\left|\la V\,,\, h^{\tau/2}\ h^{-\tau/2}\varphi_\beta\ra\right|\\
&\le \textbf{A} \left\|dh^{\tau/2}\right\|_{L^2_\mu}\,\, \left\|d(h^{-\tau/2}\varphi_\beta)\right\|_{L^2_\mu}
\end{align*}
And using \eqref{estL2c} together with \pref{prop:htau}, we deduce that 
$$\left|\la \theta,\beta\ra\right|\le \textbf{A} {\frac{\tau}{\sqrt{2\tau-1}} \capa^{\frac 12}(\cU)}\ \sqrt{C} \|\beta\|_{L^2_\mu}.$$
This inequality being true for any $\beta\in \cC^\infty_c(T^*\cU)$ and any bounded open subset $\cU$, we deduce that $\theta$ is given by a locally square integrable $1-$form  and that for  any bounded open subset $\cU$:
$$\int_{\cU} |\theta|^2d\mu\le C'\textbf{A}^2 \capa(\cU)$$
where $C'$ depends only on the doubling and Poincaré constants.
Then the conclusion of follows from the implication {\bf iii)$\Rightarrow$ ii)} in \tref{thm:trace}.

\end{proof}
\appendix
\section{Trace inequalities on metric measure space}

\subsection{Setting}
We consider a metric measure space $(X,\dist,\mu)$ that is doubling i.e. for any $x\in X$ and $r<R$:
\begin{equation}\label{doubling1}
\tag{$\text{D}_{\upkappa,\upnu}$} \mu\left(B(x,R)\right)\le \upkappa \left(\frac R r\right)^\upnu \ \mu\left(B(x,r)\right),
\end{equation}
and we will note $\upgamma=\upkappa 2^{\upnu}$ so that for any $x\in X$ and $r>0$:
$$\mu\left(B(x,2r)\right)\le \upgamma\, \mu\left(B(x,r)\right).$$
We introduce the notation:  $$V(x,r):=\mu\left(B(x,r)\right).$$
It is classical to prove that the doubling condition \eqref{doubling1} implies the following control on the volume of the balls:
$$\forall x,y\in X,\forall r>0\colon V(x,r)\le \upkappa \left(1+\frac{\dist(x,y)}{r}\right)^\upnu \, V(y,r).$$
We also assume that for some $x\in X$:
$$\int_1^{+\infty} \frac{d\tau}{V(x,\tau)}<+\infty.$$
 In that case, it is easy to check that for any $y\in X$ $$\int_1^{+\infty} \frac{d\tau}{V(y,\tau)}<+\infty.$$
And we  then introduce the operator 
$$\cK\colon L^1_{c}(X,\mu)\rightarrow  L^1_{\loc}(X,\mu)$$ defined by
$$\cK f(x)=\int_0^{+\infty} \left(\fint_{B(x,\tau)} fd\mu\right)\ d\tau$$ 
and whose Schwartz kernel is 
$$K(x,y)=\int_{\dist(x,y)}^{+\infty} \frac{d\tau}{V(x,\tau)}.$$
The dual operator $\cK^*$ has Schwartz kernel given by
$$K^*(x,y)=K(y,x)=\int_{\dist(x,y)}^{+\infty} \frac{d\tau}{V(y,\tau)}.$$

The following lemma is a consequence of the doubling hypothesis:
\begin{lem}\label{Lem:Ks} For any $x,y\in X\colon$
$$\upgamma^{-1}\,  K(y,x)\le K(x,y)\le  \upgamma \,  K(y,x).$$
Moreover  we have
$$K_{1/2}(x,y)=\int_{\frac12\dist(x,y)}^{+\infty} \frac{d\tau}{V(x,\tau)}\le \, \upgamma\, K(x,y).$$
\end{lem}

The next estimate is merely an adaptation of \cite[Lemma 2.1]{Verbsurvol}
\begin{lem}\label{Lem:KK}For any non negative $f\in L^1_{c}(X,\mu)$ and for any $x\in X$:
$$\left|\cK f\right|^2(x)\le 2\upgamma^3 \left(\cK \left(f\cK f\right)\right)(x).$$
\end{lem}
\proof
We start by the formula
$$\left|\cK f\right|^2(x)=\int_{X\times X} K(x,y)K(x,z)f(y)f(z)\, d\mu(y)d\mu(z).$$
We write $X\times X=\cU_x\cup \cV_x$ where
$$\cU_x:=\left\{ (y,z)\in X\times X,\ 2\dist(x,z)\ge \dist(y,z)\right\}\text{ and } \cV_x:=\left\{ (y,z)\in X\times X,\ 2\dist(x,z)<\dist(y,z)\right\}.$$
Notice that when $(y,z)\in \cV_x$ then
$$\dist(x,y)\ge \dist(y,z)-\dist(z,x)\ge \frac12 \dist(y,z),$$ hence
$$\cV_x\subset \cW_x=\left\{ (y,z)\in X\times X, 2\dist(x,y)\ge \dist(y,z)\right\}=\left\{ (y,z)\in X\times X, (z,y)\in \cU_x\right\}.$$
Using \lref{Lem:Ks}, one gets that for $(y,z)\in \cU_x\colon$
$$K(x,z)\le \upgamma K(z,x)\le \upgamma K_{\frac12}(z,y)\le \upgamma^2\, K(z,y)\le  \upgamma^3 \, K(y,z), $$ we obtain that:
\begin{align*}
\int_{\cU_x} K(x,y)K(x,z)f(y)f(z)\, d\mu(y)d\mu(z)&\le  \upgamma^3 \int_{\cU_x} K(x,y)K(y,z)f(y)f(z)\, d\mu(y)d\mu(z)\\
&\le  \upgamma^3  \left(\cK \left(f\cK f\right)\right)(x).
\end{align*}
Using the same argument for the integral over $\cV_x\subset \cW_x$, we obtain the result.\endproof
\subsection{Equivalent conditions for a Trace inequality}
The following theorem is a generalisation of results by Maz'ya and Verbitsky \cite{MVarkiv} and of Kerman-Sawyer \cite{KS} (see also \cite{PW,SawyerW,SWZ}) and our proof is inspired by the one provided by Verbitsky \cite{Verbsurvol}
\begin{thm}\label{thm:Genetrace} If  $(X,\dist,\mu)$ is a metric measure space satisfying the doubling property \eqref{doubling1} and $q$ is a non negative locally integrable function then the following are equivalent
\begin{enumerate}[i)]
\item The operator $\cK\sqrt{q}$ has a continuous extension $\cK\sqrt{q}\colon L^2(X,\mu)\rightarrow  L^2(X,\mu)$
\item $Q:=\cK(q)$ is finite $\mu$-almost everywhere  and the operator $\cK Q$ has a continuous extension $\cK Q\colon L^2(X,\mu)\rightarrow  L^2(X,\mu)$,
 \item $Q:=\cK(q)$ is finite $\mu$-almost everywhere and there is a constant $C$ such that 
$$\cK(Q^2)(x)\le C\cK(q)(x), \mu \text{ a.e. } x\in X,$$
\item  there is a constant $C'$ such that for any ball $B\subset X\colon$
$$\int_B \left|\cK(q\un_B)\right|^2d\mu\le C'\int_B q\,d\mu.$$
\end{enumerate} 

Moreover $\left\| \cK\sqrt{q}\right\|^2_{L^2\to L^2},\ \left\|\cK Q\right\|_{L^2\to L^2}$ and the constants $C, \,C'$ are mutually controlled.
\end{thm}
\
\begin{proof}
\vskip0.3cm
\par
{\bf Proof of  i)$\Rightarrow$ iv).} We test the $L^2$ boundedness of $T=\cK\sqrt{q}$ on $f:=\sqrt{q}\un_{B}$ and get that
$$\int_B  \left|\cK(q\un_B)\right|^2\,d\mu\le \int_X  \left|\cK(q\un_B)\right|^2\,d\mu(y)\le \|T\|^2\, \int_B  q\,d\mu.$$
Hence {\bf iv)} holds with $C'=\|T\|^2.$
\par
{\bf Proof of  ii)$\Rightarrow$ i).} We have that for any $f\in L^2(X,\mu)\colon$
\begin{align*}
\int_X q\left|\cK(f)\right|^2\,d\mu&\le 2\upgamma^3\,\int_X q\,\left(\cK \left(|f|\cK |f|\right)\right)\,d\mu\\
&\le 2\upgamma^3\,\langle \cK^*(q)\cK |f|, |f|\rangle_{L^2_\mu}\\
&\le2\upgamma^4\,\left\|Q \cK |f|\right\|_{L^2_\mu}\| f\|_{L^2_\mu}\\
&\le 2\upgamma^5\,\left\|Q \cK^* |f|\right\|_{L^2_\mu}\| f\|_{L^2_\mu}\\
&\le 2\upgamma^5\,\left\|Q \cK^* \right\|_{L^2\to L^2}\| f\|^2_{L^2_\mu}\\
&\le 2\upgamma^5\,\left\| \cK Q \right\|_{L^2\to L^2}\| f\|^2_{L^2_\mu}.
\end{align*}
Using that $\int_X q\,\left|\cK^*(f)\right|^2\,d\mu\le \upgamma^2 \int_X q\,\left|\cK(f)\right|^2\,d\mu$ we obtain that
$$\left\| \cK\sqrt{q}\right\|^2_{L^2\to L^2}=\left\| \sqrt{q}\cK^*\right\|^2_{L^2\to L^2}\le 2\upgamma^7\,\left\| \cK Q \right\|_{L^2\to L^2}. $$

{\bf Proof of  iii)$\Rightarrow$ ii).} Assuming that {\bf iii)} holds and proceeding in the same way, for any $f\in L_c^2(X,\mu)$, we obtain the inequality
\begin{align*}\int_X Q^2\left|\cK(f)\right|^2\,d\mu&\le 2\upgamma^3\,\langle \cK^*(Q^2)\cK |f|, |f|\rangle_{L^2_\mu}\\
&\le 2\upgamma^4\,\, \langle \cK(Q^2)\cK |f|, |f|\rangle_{L^2_\mu}.
\end{align*}
Using the hypothesis $ \cK(Q^2)\le C Q$ one gets that 
\begin{align*}\int_X Q^2(x)\left(\cK|f|\right)^2(x)d\mu(x)&\le  2\upgamma^4\,C\, \langle Q\cK |f|, |f|\rangle_{L^2_\mu}\\
& \le 2\upgamma^4\,C\ \| Q\left(\cK|f|\right)\|_{L^2_\mu}\, \| f\|_{L^2_\mu}. \end{align*}
As $f\in L_c^2(X,\mu)$ and  $Q\cK |f|\in L_{loc}^2$, the first inequality above implies that $Q\left(\cK|f|\right)\in L^2_\mu$.
Hence
$$\left\|Q \left(\cK|f|\right)\right\|_{L^2_\mu}\le  2\upgamma^4\,C\| f\|_{L^2_\mu}$$ and
$$\left\| \cK Q\right\|_{L^2\to L^2}=\left\| Q \cK^*\right\|_{L^2\to L^2}\le\upgamma \,  \left\| Q \cK\right\|_{L^2\to L^2}\le 2\upgamma^5\,C.$$

{\bf Proof of  iv)$\Rightarrow$ iii).} This is the most complicated implication. Our goal is therefore to estimate the quantity:
$$\cK(Q^2)(o)=\int_0^{+\infty}\left(\fint_{B(o,r)} |\cK(q)|^2(x)d\mu(x)\right)dr.$$

We decompose
$$\cK(Q^2)(o)\le 2\textbf{I}+2\textbf{II}$$
where
$$\textbf{I}=\int_0^{+\infty}\left(\fint_{B(o,r)} \left|\cK(\un_{B(o,2r)}q\,\right)|^2(x)d\mu(x)\right)dr$$ 
and
$$\textbf{II}=\int_0^{+\infty}\left(\fint_{B(o,r)} \left|\cK(\un_{X\setminus B(o,2r)}q\,\right)|^2(x)d\mu(x)\right)dr.$$
Let's start by the estimation of $\textbf{I}$,
we easily get:
$$\fint_{B(o,r)} \left|\cK(\un_{B(o,2r)}q\,\right)|^2(x)d\mu(x)\le \frac{V(o,2r)}{V(o,r)}\fint_{B(o,2r)} \left|\cK(\un_{B(o,2r)}q\,\right)|^2(x)d\mu(x)$$
and using  the property $iv)$ one gets that 
\begin{equation}\label{estimI}
\textbf{I} \le \upgamma C'\,\int_{0}^{+\infty}\left(\fint_{B(o,2r)} q(x)d\mu(x)\right)dr=\frac12 \upgamma C'\, \cK(q)(o).\end{equation}

For estimation of $\textbf{II}$, we first need an estimate of $\fint_B q$ for any ball $B=B(y,s)$. If $z,w\in B$, then $d(w,z)\le 2s$ hence 
$$K(z,w)\ge \int_{2s}^{3s} \frac{d\tau}{V(z,\tau)}\ge \frac{s}{V(z,3s)}\ge \frac{s}{\upgamma V(y,3s)}\ge  \frac{s}{\upgamma^3 V(y,s)}.$$
Hence the condition $\text{iv)}$ yields that 
$$V(y,s) \left(\frac{s}{\upgamma^3 V(y,s)}\right)^2\left(\int_B q\,d\mu\right)^2\le C' \int_B q\,d\mu$$ so that 
\begin{equation}\label{qB}
\fint_B qd\mu \le \upgamma^6 C' \frac{1}{s^2}
.\end{equation}

If $x\in B(o,r)$ and $\tau\le r$ then $B(x,\tau)\subset B(o,2r)$ so that 
 
 $$\cK\left(\un_{X\setminus B(o,2r)}q\,\right)(x)=\int_{r}^{+\infty} \left(\fint_{B(x,\tau)} \un_{X\setminus B(o,2r)}q\,d\mu\right) d\tau.$$
 Similarly if  $x\in B(o,r)$ and $\tau\ge  r$ then $B(x,\tau)\subset B(o,2\tau)$ and using  
 $V(o,2\tau)\le \upgamma^2V(x,\tau)$ one gets that 
 \begin{equation}\label{qBc}
 \cK\left(\un_{X\setminus B(o,2r)}q\,\right)(x)\le  \upgamma^2\int_{r}^{+\infty} \left(\int_{B(o,2\tau)\setminus B(o,2r)}q\,d\mu\right) \frac{d\tau}{V(o,2\tau)}.
 \end{equation}
 We introduce $\Phi(\tau)=\fint_{B(o,2\tau)}q\,d\mu,$
 Using  \eqref{qB}, we  gets that 
\begin{equation}\label{qBcc}
 \Phi(\tau)\le  \upgamma^6C'\frac{1}{\tau^2}.
 \end{equation}
 And from \eqref{qBc} we get the following estimates:
\begin{align*}
\text{II}&\le \upgamma^4\ \int_{0\le r\le\sigma\,,\, 0\le r\le \tau } \Phi(\tau)\Phi(\sigma)d\tau d\sigma dr\\
&\le 2\upgamma^4  \int_{0\le r\le\sigma\le \tau } \Phi(\tau)\Phi(\sigma)d\tau d\sigma dr\\
&\le  2\upgamma^{10}C'  \int_{0\le r\le\sigma } \Phi(\sigma)\frac{1}{\sigma}d\sigma dr\\
&\le 2\upgamma^{10}C' \int_{0 }^{+\infty}\Phi(\sigma)\left( \int_0^\sigma dr\right)\frac{d\sigma}{\sigma}\\
&\le 2\upgamma^{10}C'  \int_{0 }^{+\infty}\Phi(\sigma)d\sigma\\
&\le 2\upgamma^{10}C' \int_{0 }^{+\infty}\left(\fint_{B(o,2\sigma)}q\,d\mu\right)d\sigma\\
&\le\upgamma^{10}C'\, \cK(q)(o). \end{align*}

And this estimate together with \eqref{estimI} yields that if {\bf iv)} holds then  {\bf iii)} holds with
$$C=\left(\gamma+\upgamma^{10}\right) C'.$$

\end{proof}

\subsection{} The next estimate provides a class of operators whose Schwartz kernels look like the same as the one of the operator $\cK$ studied previously.
\begin{prop}\label{prop:estikernel} If $(X,\dist,\mu)$ is a metric measure space satisfying the doubling property \eqref{doubling1} then
for any $D,s>0$, there are  positive constants $c,C$ depending only on $D,s$ and the doubling constants $\upkappa,\upnu$ such that for any $x,y\in X\colon$
$$c \int_{d(x,y)}^\infty  \frac{r^{2s-1}}{V(x,r)}dr\le \int_0^\infty \frac{e^{-\frac{d^2(x,y)}{D t}}}{V(x,\sqrt{t})}\, t^{s-1}dt\le C\int_{d(x,y)}^\infty  \frac{r^{2s-1}}{V(x,r)}dr.$$
\end{prop}
\proof Firstly, we easily get that 
$$2 e^{-\frac{1}{D}}\int_{d(x,y)}^\infty  \frac{r^{2s-1}}{V(x,r)}dr\le \int_{d^2(x,y)}^\infty \frac{e^{-\frac{d^2(x,y)}{D t}}}{V(x,\sqrt{t})}\, t^{s-1}dt\le 2 \int_{d(x,y)}^\infty  \frac{r^{2s-1}}{V(x,r)}dr.$$
Secondly using the doubling assumption, we get 
\begin{align*}
 \int_0^{d^2(x,y)} \frac{e^{-\frac{d^2(x,y)}{D t}}}{V(x,\sqrt{t})}\, t^{s-1}dt& \le\frac{\upkappa}{V(x,2d(x,y))}
 \int_0^{d^2(x,y)} e^{-\frac{d^2(x,y)}{D t}}\left(\frac{2d(x,y)}{\sqrt{t}}\right)^\upnu\,t^{s-1}dt\\
 &\le \frac{\upkappa 2^\upnu d^{2s}(x,y)}{V(x,2d(x,y))}
 \int_0^{1} e^{-\frac{1}{D t}}\,t^{s-\upnu/2-1}dt.
\end{align*}
But we also have the lower estimate:
\begin{align*}\int_{d(x,y)}^\infty  \frac{r^{2s-1}}{V(x,r)}dr&\ge \int_{d(x,y)}^{2d(x,y)}  \frac{r^{2s-1}}{V(x,r)}dr\\
&\ge \frac{d^{2s}(x,y)}{2\,V(x,2d(x,y))}.
\end{align*}
Hence the result with 
$$c=2e^{-\frac{1}{D}}\text{  and  } C=2+\upkappa 2^{\upnu+1}\int_0^{1} e^{-\frac{1}{D t}}\,t^{s-\upnu/2-1}dt.$$
\endproof
\bibliographystyle{alpha} 
\bibliography{MV.bib}
\end{document}